\documentclass{article} 

\usepackage[utf8]{inputenc}
\usepackage[T1]{fontenc}
\usepackage[english]{babel}
\usepackage[letterpaper,top=2cm,bottom=2cm,left=3cm,right=3cm,marginparwidth=1.75cm]{geometry}

\usepackage{todonotes}
\usepackage{amsmath}
\usepackage{amsthm}
\usepackage{amsfonts}
\usepackage{amssymb}
\usepackage{comment}
\usepackage{authblk}

\usepackage{graphicx}
\usepackage{float}
\usepackage[colorlinks=true, allcolors=blue]{hyperref}
\usepackage{algorithm}
\usepackage{algpseudocode}
\usepackage{mathtools}
\numberwithin{equation}{section}
\mathtoolsset{showonlyrefs=true}
\usepackage[normalem]{ulem}
\usepackage{mathrsfs}  

\usepackage{xcolor}
\definecolor{ao}{rgb}{0.0, 0.5, 0.0}

\usepackage{tikz}
\usetikzlibrary{arrows,calc}
\tikzset{
	>=stealth', 
	help lines/.style={dashed, thick}, 
	axis/.style={<->},
	important line/.style={thick},
	connection/.style={thick, dotted},
}
\usepackage[]{pgfplots}
\pgfplotsset{compat=1.18}
\usepgfplotslibrary{statistics}
\usepackage{orcidlink}

\makeatletter
\newcommand{\subjclass}[2][1991]{%
  \let\@oldtitle\@title%
  \gdef\@title{\@oldtitle\footnotetext{#1 \emph{MSC Codes.} #2}}%
}
\newcommand{\keywords}[1]{%
  \let\@@oldtitle\@title%
  \gdef\@title{\@@oldtitle\footnotetext{\emph{Key words.} #1.}}%
}
\makeatother

\newcommand{\eps}{\varepsilon}

\newcommand{\R}{\mathbb{R}}
\newcommand{\N}{\mathbb{N}}
\newcommand{\IR}{\mathbb{R}}

\newcommand{\bg}{\boldsymbol{g}}
\newcommand{\bm}{\boldsymbol{m}}

\newcommand{\bH}{\boldsymbol{H}}
\newcommand{\bn}{\boldsymbol{n}}
\newcommand{\bu}{\boldsymbol{u}}

\newcommand{\bv}{\boldsymbol{v}}
\newcommand{\bV}{\boldsymbol{V}}
\newcommand{\bx}{\boldsymbol{x}}
\newcommand{\bX}{\boldsymbol{X}}
\newcommand{\by}{\boldsymbol{y}}
\newcommand{\be}{\boldsymbol{e}}

\newcommand{\bphi}{\boldsymbol{\phi}}
\newcommand{\bvarphi}{\boldsymbol{\varphi}}
\newcommand{\bpsi}{\boldsymbol{\psi}}
\newcommand{\bnu}{\boldsymbol{\nu}}

\newcommand{\bGamma}{\boldsymbol{\Gamma}}

\newcommand{\CC}{\mathcal{C}}

\newcommand{\TT}{\mathcal{T}}
\newcommand{\KK}{\mathcal{K}}
\newcommand{\II}{\mathcal{I}}
\newcommand{\NN}{\mathcal{N}}
\newcommand{\MM}{\mathcal{M}}
\newcommand{\PP}{\mathcal{P}}
\newcommand{\HH}{\mathcal{H}}

\newcommand{\YY}{\mathcal{Y}}

\newcommand{\rmd}{\textrm{d}}

\newcommand{\norm}[1]{\left \lVert #1 \right \rVert}
\newcommand{\snorm}[1]{\left \lvert #1 \right \rvert}
\newcommand{\dotp}[2]{\left ( #1,#2 \right )}

\def\seminorm#1#2{\left\vert #1\right\vert_{#2}}
\def\set#1{\left\{#1\right\}}
\def\setc#1#2{\left\{#1:#2\right\}}

\DeclareMathOperator*{\argmin}{arg\,min}

\DeclareMathOperator*{\erf}{erf}

\newcommand{\bHeff}{\mathbf{H}_{\text{eff}}}
\newcommand{\hC}{\hat{\CC}}

\newtheorem{remark}{Remark}

\newtheorem{proposition}{Proposition}[section]

\title{A Reduced Basis Method for the Stochastic Landau-Lifshitz–Gilbert Equation}

\author[1]{Andrea Scaglioni \orcidlink{0000-0002-8601-5902}}
\author[2]{Michael Feischl}
\author[2]{Fernando Henr\'iquez}
\affil[1]{Faculty of Mathematics, University of Vienna. Oskar-Morgenstern-Platz 1, 1090 Vienna, Austria\\andrea.scaglioni@univie.ac.at}
\affil[2]{Institute for Analysis and Scientific Computing, Vienna University of Technology, Wiedner Hauptstra{\ss}e 8-10, A-1040 Wien, Austria\\
       \{michael.feischl,fernando.henriquez\}@asc.tuwien.ac.at}



\begin{document}

\maketitle

\begin{abstract}
In this work, we consider the construction of efficient surrogates for the stochastic version of the Landau-Lifshitz-Gilbert (LLG) equation using model order reduction techniques, in particular, the Reduced Basis (RB) method.
The Stochastic LLG (SLLG) equation is a widely used phenomenological model for the time evolution of the magnetization field confined to a ferromagnetic body while taking into account the effect of random heat perturbations. This phenomenon is mathematically formulated as a nonlinear parabolic problem, where the stochastic component is represented as a parameter-dependent datum depending on a non-compact and high-dimensional parameter.
In an \emph{offline} phase, we use Proper Orthogonal Decomposition (POD) on high-fidelity samples of the unbounded parameter space. To that end, we use the so-called \emph{tangent plane scheme}. 
For the \emph{online} phase of the RB method, we again employ the tangent plane scheme in the RB space. This is possible due to our particular construction that reduces both spaces of the magnetization and of its time derivative. Due to the saddle-point nature of this scheme, a stabilization that appropriately enriches the RB space is required. Numerical experiments show a clear advantage over earlier approaches using sparse grid interpolation.
In a complementary approach, we test a sparse grid approximation of the reduced coefficients in a purely data-driven method, exhibiting the weaknesses of earlier sparse grid approaches, but benefiting from increased stability.
\end{abstract}

\section{Introduction}
The Landau-Lifshitz-Gilbert equation (LLG) is a phenomenological model describing the time-evolution of the (normalized) \emph{magnetization} $\bm:[0, T] \times D\rightarrow \R^3$ of a magnetic body of sub-micrometer length-scale.
It is of particular relevance for material science and solid state physics, as it models a wide range of microscopic effects, as well as for engineering applications such as magnetic sensors, and \emph{magnetic random access memory} (MRAM) devices among others.
From a mathematical point of view, the LLG equation is a nonlinear PDE of parabolic type with a non-convex constraint, which renders the analysis of the equation and its discretization particularly challenging and often requires non-standard arguments and tools. While there are many numerical schemes available, they often come with certain CFL-type convergence conditions. The first unconditionally convergent method is the midpoint rule~\cite{Bartels2006Convergence}. To remove the necessary solution of a nonlinear system in each time-step, the so-called \emph{tangent plane scheme} (TPS) was introduced in~\cite{Alouges2008} and extended to high-order in~\cite{Akrivis2021Higher}.

The \emph{Stochastic} LLG equation (SLLG)~\cite{brown1963thermal, kubo1970brownian} is a stochastic PDE model extending the LLG equation by a multiplicative Wiener noise term that models thermal noise disturbing the magnetization.
This is particularly relevant for engineering applications, where quantifying the effect of thermal noise is vital for the design of reliable
tools such as heat-activated magnetic recording technologies~\cite{Akagi2012Thermally}.

A number of numerical methods have been developed for the space and time discretization of the SLLG equation. 
Suitable adaptations of the midpoint rule~\cite{BanasBrzezniak2014convergent} and tangent plane scheme~\cite{alouge2014asemidiscrete} aim to solve the stochastic PDE directly, while~\cite{goldys2016Finite} applies the Doss-Sussmann transform~\cite{doss, sussman} to the SLLG equation to rewrite it into an LLG equation with a random coefficient.

Particularly for computational design and optimization, fast and efficient surrogate models and the approximation of statistical quantities of interest (e.g. moments and probabilities of events) for the SLLG equation are paramount. However, only few methods are available. 
One approach is given in~\cite{an2025sparse}, where a sparse grid (high dimensional interpolation) method is applied to the SLLG equation.
Here, the SLLG equation is transformed into a parametric PDE by means of the Doss-Sussmann transform (as in~\cite{goldys2016Finite}) followed by the Lévy-Ciesielski expansion of the Brownian motion (see e.g.~\cite[Section~3.3.3]{Evans2013Introduction}).
The regularity of the parameter-to-solution map is studied as the application of an abstract framework for more general nonlinear parametric PDEs with Gaussian noise.
Moreover, this work tackles the \emph{curse of dimensionality} in the parameter
space by exploiting the sparsity of the parameter-to-solution map.
The resulting sparse grid approximation achieves an approximation rate that is independent of the number of scalar parameters. However, this can only be shown under severe assumptions on the problem at hand and these restrictions show in numerical experiments. While experiments that do not involve strong forcing work extremely well, even moderate oscillations in time bring the sparse grid method quickly to its limits. This is explainable by the fact that sparse grid interpolation is purely data-driven and does not obey physical energy principles between the interpolation points.
The goal of this work to develop a fast surrogate model that also works in these practically relevant cases.

\subsection{Model Order Reduction and the Reduced Basis Method}
Model Order Reduction (MOR) techniques seek to construct fast and efficient surrogates for parametric maps while maintaining a certified level of accuracy with respect to the high-fidelity solution. 
Among them, the Reduced Basis Method has been successfully applied to 
stationary parametric PDEs \cite{hesthaven2016certified,quarteroni2015reduced}.
Loosely speaking, the RB method follows a two-phases, \emph{offline} and \emph{online}, paradigm. In the former, a basis of reduced dimension 
is constructed from a number of expensive high-fidelity snapshots,
whereas in the latter, for a new parametric input, a solution is obtained
as an element of the previously computed reduced space. 

The temporal variable introduces additional complexities.
The solution of the time-dependent parametric PDEs is not an element of a discretization space, 
but an entire \emph{trajectory} or \emph{sample path} over a period of time.
In addition, and contrary to stationary problems, transport-dominated problems suffer from the so-called Kolmogorov barrier~\cite{arbes2025kolmogorov, greif2019decay}, which renders traditional linear MOR techniques computationally expensive. 

In the last few years, a variety of approaches have been proposed to tackle
the plethora of challenges posed by time-dependent parametric PDEs (pPDE), we mention non exhaustively:
Combined POD-Greedy sampling techniques \cite{haasdonk2013convergence,haasdonk2008reduced},  structure-preserving approaches for Hamiltonian formulations \cite{afkham2017structure},
non-intrusive methods \cite{audouze2013nonintrusive}, and non-linear MOR
\cite{peherstorfer2022breaking}. 
We refer to \cite{hesthaven2022reduced} for a comprehensive survey
of these and other MOR techniques for time-dependent problems.

\subsection{Contributions}
In this work, we explore the construction of efficient surrogates for the 
approximation of the stochastic-parametric version of the LLG equation using the RB method. For the construction of the reduced space in the offline phase, we follow a rather standard approach. That is, we define sample points in the unbounded parameter defining the SLLG equation. Then, we numerically compute
the trajectories, sample paths or high-fidelity samples. Using these and the well-known POD method, we extract a basis of reduced dimension which captures the behavior of the solution across the previously computed samples. 
Alternatively, for time-dependent problems, usually the POD-Greedy approach is preferred~\cite{haasdonk2013convergence,hesthaven2022reduced,siena2023introduction}. 
Therein, sample paths in time are selected sequentially in a greedy fashion
by exploring the parameter space over a suitably \emph{a priori} selected training set. For practical purposes, this requires, however, the existence of
computable and reliable \emph{a posteriori} error estimators as the ones in \cite{hesthaven2016certified}, which for the SLLG are not readily available. 

In the online phase, and once the reduced basis is already at hand, we explore two different approaches for the efficient computation of 
the reduced solution for a given parametric input.
\begin{itemize}
    \item[(i)] {\bf Galerkin POD-Tangent Plane Scheme (POD-TPS).}
    For the computation of the reduced solution, we consider the discretization of the original SLLG equation in the reduced space, as opposed to a Galerkin formulation in the finite element space. As in the computation of the high-fidelity samples, the problem's evolution in time is handled using the tangent plane scheme.
    This discretization results in a saddle point problem at each time step, introducing a Lagrange multiplier to enforce the conditions required by the tangent plane scheme.
    As discussed in \cite{Gerner2012Certified}, even if the high-fidelity model is inf-sup stable, one cannot guarantee that the reduced problem inherits this property. This may lead to a loss of accuracy of the reduced basis method in the online phase.
    Following ideas from \cite{Gerner2012Certified}, we propose an approach to enrich the reduced space for the velocity field—namely, the derivative of magnetization—and obtain a provably inf-sup stable pair of reduced spaces for the velocity field and the Lagrange multiplier.
    \item[(ii)] {\bf Sparse Grid-Reduced Basis Projection (SG-RBP)}.
    We consider the parametric maps that consist of the projection of
    the parametrically-dependent magnetization field into the reduced space. 
	Next, we consider the task of approximating these coefficients, which are as many as the chosen dimension of the reduced space, by means of tailored interpolation techniques. 
    Recently, in \cite{an2025sparse,scaglioni2024thesis}, the theoretical foundations and
    computational algorithms for this sort of sparse grid interpolation
    over a high-dimensional, unbounded parameter space are established.
    Unlike the previous approach, we are not relying on a Galerkin discretization
    in the reduced spaces to compute the reduced solution in the online phase. 
    Therefore, no unstable behavior in the computation of the reduced solution is encountered. 
\end{itemize}

\subsection{Outline}
This work is structured as follows. 
In Section~\ref{sec:model_LLG}, we introduce an equivalent random-parametric coefficient LLG equation arising from the discretization of the SLLG equation.
In Section~\ref{sec:tangent_plane_scheme}, we introduce the tangent plane
scheme, which serves throughout this work as the time-stepping method for the parametric LLG equation. 

In Section~\ref{sec:RB_LLG}, we introduce the reduced basis method for the SLLG equation, with a particular focus on the construction of the reduced space using POD in the offline phase and the Galerkin projection-based reduced basis method for the computation of the reduced solution in the online phase.
As previously discussed, one cannot guarantee inf-sup stability of the reduced problem. We propose a stabilization technique, which relies on the construction of a reduced space not only for the velocity field and the Lagrange multiplier, but also for the magnetization field itself. 

Section~\ref{sec:sg} proposes a different approach for the computation of the reduced solution that does not rely on the numerical approximation of a Galerkin problem. Instead of the parameter-to-solution map, that is a map with values in the high-fidelity discretization space, we consider {its projection onto} the reduced space, which is constructed exactly as in Section~\ref{sec:RB_LLG}. Then, one only needs to approximate as many coefficients as the dimension of the reduced space as opposed to as many as the dimension of the high-fidelity discretization space.

In Section~\ref{sec:numerical_results}, we present a series of numerical experiments comparing the performance of the reduced basis method during the offline phase and in the online phase, in the latter stage for both the previously described approaches.

\section{Model Problem: A Parametric LLG equation}
\label{sec:model_LLG}
In this section, we introduce a parametric time-dependent PDE arising from the discretization of the SLLG equation (a stochastic PDE).

Let $ I \coloneqq [0,T]$ denote the time interval of interest with $T>0$, $D \subset \R^3$ a bounded Lipschitz domain and let $(\Omega, \mathcal{F}, \mathcal{P})$ a probability triple.

The parametric equation is obtained by first applying the Doss-Sussmann transform~\cite{doss, sussman} to the stochastic PDE solution. The resulting random field solves an equivalent \emph{random coefficient} LLG equation (see~\cite{an2025sparse} for more details):
\begin{align}\label{eq:pLLG}
	\begin{split}
        \begin{cases}
		\alpha\partial_t \bm + \bm\times\partial_t\bm &= (1+\alpha^2)\left[\bHeff - (\bHeff\cdot\bm)\bm\right] 
		\qquad \textrm{in } I\times D, \\
		\partial_{\bn} \bm &= \mathbf{0}\qquad \textrm{on } I\times \partial D,\\
		\bm(0) &= \bm^0\qquad \textrm{on } D.
	\end{cases}
	\end{split}
\end{align}
Here, the constant $\alpha>0$ denotes the \emph{Gilbert damping parameter}, which is an empirically derived parameter to fit observational data.
The initial condition $\bm^0:D\rightarrow \R^3$ is assumed to be unit-modulus, i.e. $\seminorm{\bm^0}{2}=1$ a.e. in $D$ with $\partial_{\bn}\bm^0 = 0$ on $\partial D$. 
The \emph{effective field} $\bHeff$ is a random field defined as:
\begin{align*}
	\bHeff(\omega, \bm) &\coloneqq \Delta \bm + \hC(W(\omega), \bm)\\
	\hC(W, \bm) &\coloneqq e^{WG} \Delta e^{-WG} \bm - \Delta \bm + e^{-WG} \bH_{\text{ext}},
\end{align*} 
where $\omega\in\Omega$, $W = W(\omega, t): \Omega\times I \rightarrow \R$ denotes the Brownian motion, $G\bu:= \bu\times \bg$ for any $\bu\in \R^3$ and $\bg:D\rightarrow \R^3$ is a datum representing the spatial noise distribution (we assume it for simplicity to be pointwise unit-modulus in $D$ with $\partial_{\bn} \bg = 0$ on $\partial D$), $e^{WG}$ denotes the exponential of the linear operator $\bu \mapsto WG(\bu)$, and 
$\bH_{\text{ext}}: I\times D\rightarrow \R^3$ represents a given external magnetic field.

In order to derive the \emph{parametric coefficient} LLG equation, we consider the Lévy-Ciesielski construction of the Wiener process, which gives the identity
\begin{align}\label{eq:LC}
	W(\omega, t) = \sum_{n=1}^{\infty} Y_n(\omega) \eta_n(t)\qquad \text{a.e. } \omega\in\Omega, \quad \forall t\in I,
\end{align}
where $(Y_n)_{n\in\N}$ are i.i.d. standard Normal random variables, and 
$(\eta_n)_{n\in\N}$ is the \emph{Faber-Schauder basis} over $I$. Here, convergence holds uniformly in $t$ and almost surely. See~\cite[Section~3.3.3]{Evans2013Introduction} for more details. 

We then substitute the random variables $Y_n$ with parameters $y_n\in\R$, to obtain:
\begin{align}\label{eq:LCP}
	W(\by, t) = \sum_{n=1}^{\infty} y_n \eta_n(t),
\end{align}

A sufficient condition for~\eqref{eq:LCP} to converge in the H\"older space $C^{\delta/2}(I)$, $0 \leq \delta < 1$ (the time regularity of the Wiener process sample paths), is that the sequence $\by = (y_n)_{n\in\N}$ satisfies:
\begin{align}\label{eq:summability_Gamma}
    \sum_{\ell\in\N_0} 
		\max_{j=1,\dots, \lceil2^{\ell-1}\rceil} 
		\seminorm{ y_{\lfloor2^{\ell-1}\rfloor+j} }{} 2^{-(1-\delta)\ell/2}	
	< \infty.
\end{align}
for any $0 \leq \delta < 1$. 
This condition is obtained by imposing $\norm{W(\by, \cdot)}_{C^{\delta}(I)}$ to be finite and using the definition of the Faber-Schauder basis (see~\cite[Section~5]{an2025sparse} for further details).

Therefore, we consider the \emph{parameter space}
\begin{align*}
	\bGamma = \set{ \by = (y_n)_{n\in\N} \in \R^{\N} : \by \text{ satisfies } \eqref{eq:summability_Gamma}}.
\end{align*}
Substituting $W(\omega, t)$ with $W(\by, t)$ for $\by\in\bGamma$ and $t\in I$, Equation~\eqref{eq:pLLG} becomes a \emph{parametric coefficient} PDE and its solution $\bm$ is a parameter-dependent function in the sense:
\begin{align*}
	\bm:\bGamma\times I\times D\to\R^3.
\end{align*}

Scalar multiplication of~\eqref{eq:pLLG} with $\bm$ shows, under the assumption of sufficient smoothness, that $\bm$ and $\partial_t\bm$ are orthogonal. Hence, the time evolution $t\mapsto \bm(\by, t, \bx)$ preserves the modulus of the initial condition, i.e.,
\begin{align*}
	\seminorm{\bm(\by, t, \bx)}{2} = \seminorm{\bm^0(\bx)}{2} \qquad \text{for all } t\in I, \by\in\bGamma, \bx\in D.
\end{align*}

This observation motivates the introduction of the \emph{tangent plane to $\bm$}, i.e.,
\begin{align}
	\KK(\bm) \coloneqq \setc{\bphi \in H^1(D)^3}{\bm\cdot\bphi =0 \textrm{ a.e. in } D},
\end{align}
and we observe $\dotp{(\bHeff\cdot\bm)\bm}{\bphi}_{L^2(D)^3} = 0$ for all $\bphi\in \KK(\bm)$. Therefore, we may introduce a simplified weak formulation by testing~\eqref{eq:pLLG} with $\bphi\in \KK(\bm)$ and integration by parts: Find $\bm \in H^1( I\times D)^3$ such that for $t\in I$ and a.e. $\by\in\bGamma$ holds
\begin{align}\label{eq:pLLG_weak}
	\begin{split}		
		\dotp{\alpha\partial_t \bm}{ \bphi}_{L^2(D)^3} + \dotp{\bm\times\partial_t\bm}{ \bphi}_{L^2(D)^3} + \dotp{\nabla \bm}{ \nabla \bphi}_{L^2(D)^3}
		= 
        \dotp{\hC(W(\by, t) \bm)}{ \bphi}_{L^2(D)^3}
	\end{split}
\end{align}
for all $\bphi=\bpsi\times \bm$ with $\bpsi \in C^\infty( I\times D)$. Note that $\bphi(t)\in \KK(\bm(t))$ for all $t\in  I$.
Due to the assumptions on $\bg$, it can be proved that the right-hand side of~\eqref{eq:pLLG_weak} admits the following representation:
\begin{align*}
	\dotp{\hC(W(\by, t), \bm)}{\bphi}_{L^2(D)^3} 
	=&
    \dotp{\nabla\bm}{\nabla\bphi}_{L^2(D)^3}
	- 
    \dotp{\nabla\left(e^{WG}\bm\right)}{\nabla\left(e^{WG}\bphi\right)}_{L^2(D)^3}\\
	&+ \dotp{e^{-WG}\bH_{\text{ext}}}{ \bphi}_{L^2(D)^3},
\end{align*} 
where in turn, for all $\bphi \in L^2(D)^3$, we have (see \cite[Equation~4.3]{an2025sparse})
\begin{align*}
	e^{WG} \bphi = \bphi + \sin(W)G \bphi + (1-\cos(W)) G^2 \bphi .
\end{align*}
Thus, Equation~\eqref{eq:pLLG_weak} is equivalent to
\begin{align} \label{eq:pLLG_weak_simpler}
	\begin{split}		
		\dotp{\alpha\partial_t \bm}{ \bphi}_{L^2(D)^3} + \dotp{\bm\times\partial_t\bm}{ \bphi}_{L^2(D)^3} + \dotp{\nabla\left(e^{WG}\bm\right)}{\nabla\left(e^{WG}\bphi\right)}_{L^2(D)^3}
		=
        \dotp{e^{-WG}\bH_{\text{ext}}}{\bphi}_{L^2(D)^3}.
	\end{split}
\end{align}

The well-posedness of the weak form of this random-parametric coefficient LLG equation has been studied extensively in~\cite{goldys2016Finite}.

\subsection{High-Fidelity Discretization: Tangent Plane Scheme}
\label{sec:tangent_plane_scheme}
The tangent plane scheme is a well-established method for the approximation of the time-evolution of the LLG equations. 
Used together with the finite element method, it provides a scheme for the full discretization of the SLLG equation~\eqref{eq:pLLG}. 
See~\cite{Akrivis2021Higher,Alouges2008,alouge2014asemidiscrete,goldys2016Finite} and the references therein for more details, convergence properties, and error estimates.
In this section, we briefly describe the method and its use to approximate~\eqref{eq:pLLG_weak} in space and time.

Consider a shape-regular mesh $\TT_h$ on $D$ with mesh size $h>0$ and denote by $\NN_h$ its vertices.
Denote by $V_h$ the space of real-valued functions on $D$ that are globally continuous and piecewise affine over $\TT_h$ and 
$\bV_h = \left[V_h\right]^3$.
Let $\Pi_h: H^1(D)^3 \rightarrow \bV_h$ denote the $L^2$-projection onto $\bV_h$.

Consider the set of normalized discrete fields
\begin{align*}
	\MM_h &\coloneqq \setc{\bphi_h\in \bV_h}{\seminorm{\bphi_h(\bx)}{2}=1\ \forall \bx\in\NN_h}
\end{align*}
and the \emph{discrete tangent plane} to $\bm_h\in\MM_h$ imposed in an $L^2(D)$-sense (introduced in~\cite{Akrivis2021Higher})
\begin{align}\label{eq:def_TP_classic}
	\KK_h(\bm_h) 
	\coloneqq 
	\setc{\bphi_h\in \bV_h}{\dotp{\bphi_h\cdot \bm_h}{ \psi}_{L^2(D)} = 0\ \forall \psi \in V_h}.
\end{align}

Note that, in the original tangent plane scheme reference~\cite{Alouges2008}, the tangent plane was defined via $\NN_h$-nodewise orthogonality. However, the error estimates in~\cite{Akrivis2021Higher} for higher-order counterparts of the tangent plane scheme required this slightly different definition. We refer to~\cite{Akrivis2021Higher} for further discussions on this topic. For the reduced basis approach below, the definition via $L^2(D)$ orthogonality is natural, as the elements of the reduced basis don't have a natural correspondence to mesh nodes.

We define the trilinear form
$\mathsf{a}(\cdot,\cdot,\cdot): H^1(D)^3 \times H^1(D)^3 \times H^1(D)^3  \rightarrow \IR$ as
\begin{equation}\label{eq:trilinear_form_a}
	\mathsf{a}(\boldsymbol{v},\boldsymbol{m},\boldsymbol{\phi})
	\coloneqq
	\dotp{\alpha \boldsymbol{v} + \boldsymbol{m}\times\boldsymbol{v}}{\boldsymbol{\phi}}_{L^2(D)^3}
	+ 
	\dotp{\nabla (\boldsymbol{m} + \tau\boldsymbol{v})}{\nabla \boldsymbol{\phi}}_{L^2(D)^3}
	\quad
	\forall 
	\boldsymbol{v},\boldsymbol{m},\boldsymbol{\phi} \in H^1(D)^3,
\end{equation}
together with the time- and parameter-dependent bilinear form
$\ell(\cdot,\cdot;\by,t): H^1(D)^3 \times H^1(D)^3 \rightarrow \IR$
\begin{equation}
	\ell(\bm,\bphi;\by,t)
	\coloneqq
	\dotp{\hC(W(\by, t), \bm)}{\bphi}_{L^2(D)^3},
	\quad
	\forall 
	\bm,\bphi \in H^1(D)^3.
\end{equation}

Given $N_T\in\N$, we define the time step size $\tau = \frac{T}{N_T}$ and time steps $t_n = n\tau$ for $n=0,\dots,N_T$.
The tangent plane scheme applied to the parametric LLG equation~\eqref{eq:pLLG} for a fixed $\by\in \bGamma$ computes approximations $\bm_h^{n}(\by) \approx \bm(\by, t_n, \cdot) \in \MM_h$ for all $n=0,\dots,N_T$.
The algorithm reads as follows:

\begin{algorithm}[H]
	\caption{Tangent Plane Scheme (TPS)}
	\label{algo:TPScheme}
	\hspace*{\algorithmicindent} {\bf Input:} $\bm^0$, $N_T$, $\by$, $\tau$\\
	\hspace*{\algorithmicindent} {\bf Output:} $\left(\bm_h^n(\by)\right)_{n=0}^{N_T}$
	\hspace*{\algorithmicindent} 
	\begin{algorithmic}[1] 
		\State Compute $\bm_h^0(\by) = \Pi_h \bm^0$ and set $\widehat{\bm}_h^{0}(\by) = \frac{\bm_h^0}{\snorm{\bm_h^0}}$
		\For{$n=0,1,...,N_T-1$}
		\State {\bf Compute velocity:} Find $\bv_h^n(\by) \in \KK_h(\widehat{\bm}_h^n(\by))$ such that
			\begin{align*}
				\mathsf{a}(\boldsymbol{v}^n_h(\by),\widehat{\bm}^n_h(\by),\boldsymbol{\phi}_h)
				= 
				\ell(\widehat{\bm}^n_h(\by),\bphi_h;\by,t_n)
				\qquad \forall \bphi_h\in \KK_h(\widehat{\bm}_h^n(\by))
			\end{align*} \label{line:eq_TPS}
		  \State {\bf Update magnetization: }
          \begin{equation}\label{line:m_update}
              \bm_h^{n+1}(\by) = \widehat{\bm}_h^n(\by) + \tau \bv_h^n(\by).
          \end{equation}
        \State {\bf Nomalize magnetization: }
        \begin{equation}\label{line:normalize}
            \widehat{\bm}_h^{n+1}(\by) = 
            \frac{\bm_h^{n+1}(\by) }{\snorm{\bm_h^{n+1}(\by)}}.
        \end{equation}
		\EndFor
	\end{algorithmic}
\end{algorithm}
We emphasize that:
\begin{itemize}
	\item The discrete tangent plane $\KK_h(\widehat{\bm}_h^n)$ depends on $\widehat{\bm}_h^n(\by)$ and hence must be updated at each time step.
	The problem can be reduced to a formulation with $\bV_h$ as test and trial
	spaces by introducing appropriate Lagrange multipliers in $V_h$:
	Find $(\bv_h^n(\by), \lambda^{n}_h) \in \bV_h \times V_h$ such that for all $(\bphi_h, \xi_h) \in \bV_h \times V_h$ it holds
	\begin{equation}
	\begin{aligned}\label{eq:saddle_point_for}
		\mathsf{a}(\boldsymbol{v}^n_h(\by),\widehat{\bm}^n_h(\by),\boldsymbol{\phi}_h)
		+ 
        \dotp{\lambda_h^n}{ \bphi_h\cdot \widehat{\bm}_h^n(\by)}_{L^2(D)}
		&= \ell(\widehat{\bm}_h^n(\by), \bphi_h;\by, t) \\
		\dotp{\bv^{(n)}_h\cdot \widehat{\bm}_h^n(\by)}{ \xi_h}_{L^2(D)}&= 0,
	\end{aligned}
	\end{equation}
        Alternatively, one may implement the tangent plane scheme for test and trial spaces via a Gram Schmidt-like procedure. See e.g.~\cite{Kraus2019Interative}. We note, however, that the saddle point formulation transfers directly to non-standard basis function as in the reduced basis approach below.
	\item Observe that the algorithm guarantees $\left(\bm_h^{n}\right)_{n=0}^{N_T} \subset \bV_h$ but does not compute a normalized approximation.  This version of the algorithm may be preferred for the theoretical analysis, because of nonlinear nature of the normalization. We refer to~\cite{BartelsProjFree} for details on the analysis.
    However, physicists often prefer the normalized version of the algorithm as it typically leads to more plausible results on coarser time discretization.
    In this case, we could directly normalize the magnetization update:
	\begin{align*}
		\bm_h^{n+1} = \frac{\bm_h^n + \tau \bv_h^n}{\snorm{\bm_h^n + \tau \bv_h^n}} \in \MM_h
	\end{align*}
        and let $\widehat{\bm}_n^n(\by) = \bm_h^n(\by)$.
	\item The tangent plane can alternatively be defined with a nodal interpolator rather than an $L^2(D)$-projection. Higher-order convergence however has not yet been proven for this variant, although there is some progress for the related harmonic map heat flow~\cite{nodalconstraints}.
	\item As proved in~\cite{BartelsProjFree}, the scheme is unconditionally stable in the sense that no relation between time step size and mesh size is necessary for stability.
\end{itemize}

Finally, we may interpolate in time the discrete values $\left(\bm_h^n\right)_{n=0}^{N_T} \subset \bV_h$ to obtain a continuous-in-time discrete solution $\bm_{\tau, h}:\bGamma\times I\times D \rightarrow \R^3$ where, for each $\by\in\bGamma$, $t\in I$, $\bm(\by, t, \cdot ) \in \bV_h$. 

In~\cite{Akrivis2021Higher}, the authors study the convergence of this method and higher-order generalizations applied to the deterministic LLG equation, i.e. $\bHeff = \Delta \bm$.
For order $k$ BDF time stepping and order $p$ finite elements, they obtain convergence of order $\mathcal{O}(\tau^k + h^p)$ of the $L^{\infty}(I, H^1(D)^3)$ error under a mild CFL condition for $k=1,2$.
While they also study $k=3,4,5$, we are not as interested in these higher-order schemes because of the Hölder-$\frac{1}{2}$ sample-paths regularity of the Brownian motion, which limits the time regularity of the magnetization, thus also the maximum possible approximation rate in time to $\frac{3}{2}$.

\section{The Reduced Basis Method for the Parametric LLG Equation}
\label{sec:RB_LLG}
We are interested in the efficient approximation of the parameter-to-solution map for the magnetization field, i.e., 
\begin{equation}
    \mathcal{S}_{\bm}: \bGamma \rightarrow L^2(I;H^1(D)^3):  \by \mapsto \bm_h(\by)
\end{equation}
using MOR, in particular the RB method,
see e.g. \cite{hesthaven2016certified,quarteroni2015reduced} and
\cite{hesthaven2022reduced} for a survey concerning its application to time-dependent problems.

However, the TPS scheme, as described in Section~\ref{sec:tangent_plane_scheme} computes, at each time step, an approximation of the velocity field. The magnetization field itself is recovered in a post-processing step, as done in Line~\ref{line:m_update} of Algorithm~\ref{algo:TPScheme}. Therefore, we focus on the approximation of the velocity parameter-to-solution map, that is 
\begin{equation}\label{eq:p_2_s_map}
    \mathcal{S}_{\bv}: \bGamma \rightarrow L^2(I;H^1(D)^3):  \by \mapsto \bv_h(\by).
\end{equation}
More precisely, we are interested in the construction of a finite-dimensional subspace of $\bV_h$ such that each element of the solution manifold
\begin{equation}
    \mathfrak{M}
    \coloneqq
    \left\{
        \boldsymbol{v}_h(t,\by):
        \;
        (t,\by) \in I \times \bGamma
    \right\}
    \subset
    \bV_h
\end{equation}
is efficiently approximated. 

In so doing, we encounter the following difficulties. 
First, the presence of essential nonlinearities of the form $\bm \times \Delta \bm$ precludes the rapid convergence of the reduced basis method and, for that matter, of any linear model order reduction technique. The use of the TPS from Section \ref{sec:tangent_plane_scheme}, which requires solving \emph{linear} problems only, allows us to circumvent this issue.
However, with the use of the TPS also comes the requirement of using an inf-sup stable pair of spaces to discretize the saddle point formulation in Equation~\eqref{eq:saddle_point_for}. This cannot be proved for standard choices of reduced spaces for the primal variable and the Lagrange multiplier used in the online Galerkin POD. We therefore apply a stabilization procedure consisting of enriching the primal space with suitable \emph{supremizer} functions. 
Finally, the parameter space $\bGamma$ is of high, and possibly countably infinite, dimension.
The numerical tests from Section~\ref{sec:numerics_nd} suggest that the POD generates effective reduced bases robustly with respect to the dimension of $\bGamma$. 

In Section \ref{eq:rb_POD} we introduce the construction of reduced spaces using POD at a continuous level and shortly discuss rates of convergence.
In Section \ref{eq:empirical_POD}, we present the empirical counterpart of POD, and in Section \ref{sec:rb_stab} we propose
a tailored stabilization strategy for the online phase of the Galerkin POD method.  

\subsection{Construction of Reduced Spaces using POD}\label{eq:rb_POD}
As discussed, we are interested in the construction of a reduced space  $\bV^{\text{(rb)}}_{J}$ of finite, and hopefully small, dimension $J$ for the approximation of the parameter-to-solution map \eqref{eq:p_2_s_map} such that
\begin{equation}\label{eq:rom_basis_cont}
	\bV_{J}^{(\text{rb})}
	=
	\argmin_{
	\substack{\bV_{J} \subset \bV_h\\\dim \bV_{J}\leq J}}
	\varepsilon
	(\bV_{J}),
\end{equation}
where, for any $\bV_{J} \subset \bV_h$, we define
\begin{equation}\label{eq:error_measure_continuous}
\begin{aligned}
    \varepsilon
	\left(
        \bV_{J}
	\right)
    &
	\coloneqq
	\norm{
            \bv_h- \mathsf{P}_{\bV_{J}}\bv_h
        }_{L^2_\mu(\bGamma;L^2(I;H^1(D)^3))}^2
        \\
        &
	=
	\int_{\bGamma}
	\norm{
		\bv_h(\by)
		-
		\mathsf{P}_{\bV_{J}}
		\bv_h(\by)
	}^2_{L^2(I;H^1(D)^3)}
	\mu(\text{d}\by),
\end{aligned}
\end{equation}
with $\mu$ being the tensorized standard Gaussian measure, i.e. the tensor product of standard 
Gaussian measures on $\Gamma_i = \R$ for all $i=1,2,\dots$
and  $\mathsf{P}_{\bV_{J}}: H^1(D)^3 \rightarrow \bV_{J}$ 
is the orthogonal projection operator onto $\bV_{J}$ with respect
to the $H^1(\text{D})^3$ inner product. 

The operators defined via
\begin{align}
	\mathsf{T}_h\colon& L^2_\mu(\bGamma)\to \bV_h,&
	g\mapsto
	\mathsf{T}_h g 
	&=
	\int_{\bGamma}
	\bv_h(\by) g(\by) 
	\mu(\text{d}\by),\\
	\mathsf{T}^\star_h\colon& \bV_h\to L^2_\mu(\bGamma),&
	x\mapsto
	\mathsf{T}^\star_h x
	&=
	\big(\bv_h(\by),x\big)_{H^1(D)^3}
\end{align}
are of Hilbert-Schmidt type, and thus compact. This makes the integral operator
\begin{align}\label{eq:continuousio}
	\mathsf{K}_h\colon \bV_h\to \bV_h,\qquad
	x\mapsto
	\mathsf{K}_hx
	=
	\mathsf{T}_h\mathsf{T}^\star_h x
	=
	\int_{\bGamma}
    \bv_h(\by)\big(u_h(\by),x\big)_{H^1(D)^3} \mu(\text{d}\by)
\end{align}
compact, self-adjoint, and positive definite. 
Consequently, it has a countable sequence of eigenpairs
$(\zeta_{h,i},\sigma_{h,i}^2)_{i=1}^r \in \bV_h \times\mathbb{R}_{\geq0}$,
with $r \in \mathbb{N}$ being the rank of the operator $\mathsf{T}_h$. Moreover, the eigenvalues accumulate at zero, and we assume that ${\sigma}_{h,1}\geq {\sigma}_{h,2}\geq \cdots\geq {\sigma}_{h,r}\geq 0$.
It is well-known that the span of the eigenfunctions to the $J$ largest eigenvalues,
\begin{align}\label{eq:rb_exact}
	\bV_{J}^{(\text{rb})}
	=
	\operatorname{span}
    \left\{
        \zeta_{h,1}
        ,\ldots,
        \zeta_{h,J}
    \right\}
	\subset \bV_h
\end{align}
minimizes the projection error \eqref{eq:error_measure_continuous}
in the $L^2_\mu(\bGamma;L^2(I;H^1(D)^3))$-norm
among all subspaces of $\bV_h$ 
of dimension at most $J$.
Furthermore, following \cite[Proposition 6.3]{quarteroni2015reduced}, it holds
\begin{align}\label{eq:bounderrmeas}
    \varepsilon
    \left(
        \bV_{J}^{(\text{rb})}
    \right)
    =
    \sum_{i=J+1}^r\sigma_{h,i}^2.
\end{align}

\subsection{Reduced Basis Construction using Empirical POD}\label{eq:empirical_POD}
Let $\{\bvarphi_{1},\dots,\bvarphi_{N_h}\}$ be a basis of $\bV_h$, 
and denote by boldface letters $\mathbf{w}_h\in\mathbb{R}^{N_h}$
the coefficients vector of a function $\boldsymbol{w}_h\in \bV_h$ in the
aforementioned basis.
Let ${\bf Q}_h \in \mathbb{R}^{N_h\times N_h}$ be defined as
\begin{align}
	\left(
		{\bf Q}_h
	\right)_{i,j}
	=
	\dotp{\boldsymbol\varphi_{i}}{\boldsymbol\varphi_{j}}_{H^1(D)^3},
	\quad
	i,j
	\in \{1,\dots,N_h\}.
\end{align}
This one-to-one correspondence yields finite dimensional representations
of the norm and inner product in $\bV_h$, which read
\begin{equation}
	\dotp{\boldsymbol{w}_h}{\boldsymbol{z}_h}_{H^1(D)^3}
	=
	{\bf w}_h^{\star}{\bf Q}_h{\bf z}_h
	\quad
	\text{and}
	\quad
	\norm{\boldsymbol{w}_h}_{H^1(D)^3}
	=
	\sqrt{
		{\bf w}_h^{\star}{\bf Q}_h{\bf w}_h
	}
    \eqqcolon
    \norm{{\bf w}_h}_{{\bf Q}_h}
\end{equation}
for $\boldsymbol{w}_h,\boldsymbol{z}_h \in \bV_h$ and their corresponding coefficient vectors ${\bf w}_h,{\bf z}_h \in \mathbb{R}^{N_h}$.
Observe that ${\bf Q}_h$ is symmetric and positive definite.

Let $\bV_{J}$ be a subspace of $\bV_h$ of dimension $J$ spanned
by the orthonormal basis $\{\bu_h^{(1)},\dots,\bu_h^{(J)}\}$.
Set ${\bf\Phi}=\left({\bf u}_h^{(1)},\dots,{\bf u}_h^{(J)}\right) \in \mathbb{R}^{N_h \times J}$
with coefficient vectors ${\bf u}_h^{(i)} \in \mathbb{R}^{N_h}$ corresponding to the basis
$\{\bvarphi_{1},\dots,\bvarphi_{N_h}\}$ of $\bV_h$.

Let $\bv_h(t,\by) \in \bV_h$ be the {velocity field, which depends both on time and the parametric variable} $\by \in \bGamma$. Furthermore, let
$\{t_0,\dots,t_{N_t}\} \subset I$ and $\{\by^{(1)},\dots,\by^{(N_S)}\} \subset \bGamma$
be sampling points in the time interval $I$ and the parameter space $\bGamma$, 
respectively.
For $(t, \by) \in I \times \bGamma$,
we consider ${\bf v}_h(t,\by) \in \mathbb{R}^{N_h}$ to be the representation 
of $\bv_h(t,\by) \in \bV_h$ with respect to $\{\bvarphi_{1},\dots,\bvarphi_{N_h}\}$ and define
\begin{equation}
\begin{aligned}
	\varepsilon
	\left(
		\bV_{J}
	\right)
	&
	\coloneqq
	\frac{1}{N_S N_t} \sum_{m=1}^{N_S} \sum_{n=1}^{N_t}
	\norm{
		{\bf v}_h
		\big(
			t_n,
			\by^{(m)}
		\big)
		-
		\sum_{j=1}^J
		\Big(\big({\bf u}^{(j)}_h\big)^\top{\bf Q}_h{\bf v}_h\big(t_n,\by^{(m)}\big)\Big)
		{\bf u}^{(j)}_h
	}^2_{{\bf Q}_h}
	\\
	&
	=
	\frac{1}{N_S N_t} \sum_{m=1}^{N_S}  \sum_{n=1}^{N_t}		
	\norm{
		{\bf v}_h
		\big(
			t_n,
			\by^{(m)}
		\big)
		-
		{\bf \Phi}
		{\bf \Phi}^\star
		{\bf Q}_h
		{\bf v}_h
		\big(
			t_n,
			\by^{(m)}
		\big)
	}^2_{{\bf Q}_h}.
\end{aligned}
\end{equation}
We are interested in finding the optimal subspace, i.e.,
\begin{equation}
	\bV^{\text{(rb)}}_{J}
	=
	\argmin_{\substack{\bV_{J} \subset \bV_{h} \\ \text{dim}(\bV_{J} ) \leq J}}
	\varepsilon
	\left(
		\bV_{J}
	\right),
\end{equation}
which is equivalent to finding ${\bf \Phi}^{\text{(rb)}}_J \in  \mathbb{R}^{N_h \times J}$
such that
\begin{equation}\label{eq:empirical_reduced_space_V_rb}
	{\bf \Phi}^{\text{(rb)}}_J
	=
	\argmin_{\substack{ {\bf \Phi}_J \in  \mathbb{R}^{N_h \times J}\\ {\bf \Phi}_J \in \mathscr{V}_J}}
	\varepsilon
	\left(
		{\bf \Phi}_J
	\right).
\end{equation}
Here, the notation $\varepsilon({\bf \Phi}_J)$ reflects the dependence 
on the subspace $\bV_J$ which is encoded in the columns of ${\bf \Phi}_J$, and
\begin{equation}
	\mathscr{V}_{J}
	\coloneqq
	\left\{
		{\bf \Phi} \in \IR^{N_h \times J}:
		{\bf \Phi}^\top 
		{\bf {Q}}_h
		{\bf \Phi}
		=
		{\bf I}_{J}
	\right\}.
\end{equation}

To compute the optimal subspace $\bV_J$, we first sample i.i.d. standard Normal parameter vectors $\by^{(1)}, \dots, \by^{(N_S)} \in \bGamma$. 
By this we mean that for each $m \in \{1,\dots,N_S\}$, $\by^{(m)} = (y^{(m)}_1, \dots, y^{(m)}_{s})$ for some $s<\infty$ and the components are i.i.d. standard normal. 
Then, we assemble the snapshot matrix for the trajectory stemming from the parametric sample $\by^{(m)}$ as follows
\begin{equation}
	\mathbf{S}_m
	\coloneqq
	\begin{pmatrix}
		{\bf v}_h^0(\by^{(m)}), \dots, {\bf v}_h^{N_t-1}(\by^{(m)})
	\end{pmatrix}
	\in \mathbb{R}^{N_h \times N_t}
\end{equation}
and consider as well the global snapshot matrix 
\begin{align}\label{eq:snapshot_matrix}
	\mathbf{S}
	\coloneqq
	\left(
		\mathbf{S}_1,\dots,\mathbf{S}_{N_S}
	\right)
	\in
	\IR^{N_h \times N}, 
\end{align}
where $N = N_t N_S$. 
Next, we consider the SVD $\widetilde{\mathbf{S}} = \mathbf{U}\boldsymbol{\Sigma}\mathbf{V}^{\dagger}$ of $\mathbf{S} = N^{-1/2}{\bf Q}_h^{1/2}\widetilde{\mathbf{S}}$, 
\begin{equation}
	{\bf U}
	=
	\left(
		{\boldsymbol{\zeta}}_1,
		\ldots,
		{\boldsymbol{\zeta}}_{N_h}
	\right)\in \mathbb{R}^{N_h \times N_h},
	\quad 
	{\bf V}
	=
	\left(
		{\boldsymbol{\psi}}_1,
		\ldots,
		{\boldsymbol{\psi}}_{N_h}
	\right) 
	\in 
	\mathbb{R}^{N \times N},
\end{equation}
are orthogonal matrices, and $\boldsymbol{\Sigma}=\operatorname{diag}
\left(\check{\sigma}_{h,1}, \ldots, \check{\sigma}_{h,\check{r}}\right) 
\in \mathbb{R}^{N_h \times N}$ with $\check{\sigma}_{h,1} \geq \cdots \geq \check{\sigma}_{h,\check{r}}>0$, 
being $\check{r} \in \mathbb{N}$ the rank of 
$\widetilde{\mathbf{S}}$. 
Then, the solution of \eqref{eq:empirical_reduced_space_V_rb} can be expressed using the 
SVD of $\widetilde{\mathbf{S}}$ as follows
\begin{equation}\label{eq:reduced_space}
	{\bf \Phi}^{\text{(rb)}}_J
	=
	\left(
		{\boldsymbol\zeta}^{\text{(rb)}}_1,
		\dots,
		{\boldsymbol\zeta}^{\text{(rb)}}_J
	\right)
	=
	\left(
	{\bf Q}_h^{-1/2}
	{\boldsymbol\zeta}_1,
	\dots,
	{\bf Q}_h^{-1/2}
	{\boldsymbol\zeta}_J
	\right)
\end{equation}
for $J\leq \check{r}$. 

Following \cite[Proposition 6.1]{quarteroni2015reduced}, we have that
\begin{equation}
	\varepsilon
	\left(
		\bV^{\text{(rb)}}_{J}
	\right)
    =
    \sum_{j=J+1}^{\check{r}}
    \check{\sigma}^2_{j}.
\end{equation}

\begin{remark}[Criterion to select $J$]\label{eq:criterium}
As per usual, given a tolerance $\epsilon_{\normalfont\text{POD}}>0$, 
we consider the minimum integer $J$ such that
\begin{equation}
    \frac{
        \sum_{j=1}^J
        \check{\sigma}^2_{j}
    }{
        \sum_{j=1}^{\check{r}}
        \check{\sigma}^2_{j}
    }
    \geq
    1
    -
    \epsilon^2_{\normalfont\text{POD}}.
\end{equation}  
\end{remark}

\begin{remark}\label{rmk:POD_other}
Observe that although the construction of the reduced basis has been thoroughly explained 
for the velocity field $\bv_h(t,\by)$, the exact same procedure can be applied for the construction
of a reduced space for the magnetization field and the Lagrange multiplier involved in the discretization of
\eqref{eq:saddle_point_for}.
\end{remark}

\subsection{Reduced Basis Approximation and Supremizer Stabilization}
\label{sec:rb_stab}
Following the construction of the reduced basis for the velocity field described in Section~\ref{eq:empirical_POD}, in the online phase of the Galerkin POD method, we compute the parametric solution for a given parametric  input $\by \in \bGamma$. We do so by seeking the solution of the corresponding variational problem posed in $\bV^{\text{(rb)}}_J$, i.e. the reduced space of dimension $J$ for the velocity field, as opposed to in $\bV_h$, namely the
high-fidelity discretization space. In addition, considering that in~\eqref{eq:saddle_point_for} we solve a saddle point problem, we also construct a reduced space $W^{\text{(rb)}}_R$ of dimension $R$ for the Lagrange multiplier.

Observe that the dimension of the reduced space for the Lagrange multiplier is not necessarily the same as that of the velocity field.

These considerations yield the following reduced problem to be solved in the online phase of the Galerkin POD method, which 
is referred to in the following as Galerkin POD-TPS.

\begin{algorithm}[H]
	\caption{Galerkin POD-TPS: Reduced Basis Online Phase with the TPS}
	\label{algo:POD-TPS-Reduced-Basis}
	\hspace*{\algorithmicindent} {\bf Input:} $\bm^0$, $N_T$, $\bV^{\text{(rb)}}_J$, $W^{\text{(rb)}}_R$, $\by$, $\tau$\\
	\hspace*{\algorithmicindent} {\bf Output:} $\left(\bm_{J}^n(\by)\right)_{n=0}^{N_T}$
	\begin{algorithmic}[1] 
	\State Compute $\bm_J^0(\by) = \mathsf{P}^{\text{(rb)}}_J \bm^0 \in \bV^{\text{(rb)}}_J$ and set $\widehat{\bm}^{0}_J(\by)
			=
			\frac{
				\bm^{0}_J (\by)
			}{
				\snorm{\bm^{0}_J(\by)}
			}
            \in
            \bV_h.$
	       \For{$n=0,1,...,N_T-1$}
		\State {\bf Compute velocity: }
        Find $\left(\boldsymbol{v}^{n}_J(\by),\lambda^{n}_R(\by)\right) \in \bV^{\text{(rb)}}_J \times W^{\text{(rb)}}_R$ such that
        \begin{equation}\label{eq:saddle_point_for_RB}
		\begin{aligned}
			\mathsf{a}
			\left(
				\boldsymbol{v}^{n}_J(\by),\widehat{\bm}^{n}_h(\by),\bphi_J
			\right)
            +
            \dotp{\lambda^{n}_R(\by)}{ \bphi_J \cdot \widehat{\bm}^{n}_h(\by)}_{L^2(D)}
            &
			= 
            \ell
			\left(
				\widehat{\bm}^{n}_h(\by),\bphi_J;\by,t_n
			\right),
            \\
            \dotp{\bv^{n}_J(\by)\cdot \widehat{\bm}^{n}_h(\by)}{ \xi_R}_{L^2(D)}
            &
		= 
            0,
		\end{aligned}
		\end{equation}
        for all $\left(\bphi_J,\xi_R\right)\in \bV^{\text{(rb)}}_J \times W^{\text{(rb)}}_R$.
		\State {\bf Update magnetization: }
		\begin{equation}\label{eq:update_m_RB_TPS}
			\bm^{n+1}_J(\by)
			=
            \bm^{n}_J(\by) + \tau \bv^{n}_J(\by)
            \in
            \bV^{\text{(rb)}}_J.
        \end{equation}
        \State {\bf Normalize Magnetization:}
		\begin{equation}\label{line:normalize_mag}
			\widehat{\bm}^{n+1}_{J}(\by)
			=
			\frac{
				\bm^{n+1}_J(\by)
			}{
				\snorm{\bm^{n+1}_J(\by)}
			}
            \in
            \bV_h.
        \end{equation}	
	\EndFor
	\end{algorithmic}
\end{algorithm}

The standard Galerkin POD-TPS method is not guaranteed to be inf-sup stable. 
More precisely, at each time step, the inf-sup stability of the high-fidelity model, given in our case by the TPS described in Algorithm \ref{algo:TPScheme}, does not necessarily imply inf-sup stability of $\bV^{\text{(rb)}}_J$ and $W^{\text{(rb)}}_R$  for any combination of reduced dimensions $J$ and $R$.

Following \cite{Gerner2012Certified}, to obtain a stable pair of reduced spaces, it is required to enrich the velocities' reduced space appropriately. 
An important observation is that the dependence of the saddle point problem described in Algorithm \ref{algo:POD-TPS-Reduced-Basis} upon the parametric input $\by \in \bGamma$ is through the right-hand
side of \eqref{eq:saddle_point_for_RB} and the normalized magnetization field $\widehat{\bm}^{n}_h(\by)$ computed in the 
previous iteration of the TPS.

Let $\boldsymbol{X}^{\text{(rb)}}_K \coloneqq \text{span} \left\{\boldsymbol{\eta}^{\text{(rb)}}_1,\dots,\boldsymbol{\eta}^{\text{(rb)}}_K\right\}$ be the reduced space of dimension $K$ for the magnetization field. This reduced basis is computed in the offline
phase of the Galerkin POD by taking the approximation of the normalized magnetization field computed in Line \ref{line:normalize} of Algorithm \ref{algo:TPScheme}
and applying the empirical POD method from Section \ref{eq:empirical_POD}.
We are interested in the construction of an
inf-sup stable pair of
reduced spaces for the discretization of the TPS in the online phase of the RB method. To this end, given $\bm \in H^1(D)^3$, we define the \emph{supremizer}
operator $T_{\bm}: H^1(D) \rightarrow \bV_h $ as
\begin{equation}
    \dotp{T_{\bm}\lambda_h}{\bv_h}_{H^1(D)^3}
    \coloneqq
    \dotp{\lambda_h}{\boldsymbol{v}_h \cdot \bm}_{L^2(D)},
    \quad
    \forall \boldsymbol{v}_h \in \bV_h, \lambda_h \in V_h.
\end{equation}
In addition, let $W^{\text{(rb)}}_R  = \text{span}
\left\{\zeta^{\text{(rb)}}_1,\dots,\zeta^{\text{(rb)}}_R\right\}$ be the reduced space for the Lagrange multiplier. 
We define the \emph{enriched}
reduced space for the velocity field as follows:
\begin{equation}\label{eq:stabilized_space_velocity}
    \bV^{\text{(rb)},\text{{stab}}}_{J,K,R}
    =
    \bV^{\text{(rb)}}_J
    \oplus
    \text{span}
    \left\{
        T_{\boldsymbol{\eta}^{\text{(rb)}}_k}
        \zeta^{\text{(rb)}}_r:
        1\leq k \leq K,\;
        1\leq r \leq R
    \right\}.
\end{equation}
Observe that now the space is of dimension 
$J + R \times K$, that is we add $R \times K$ elements
to the original reduced space $\bV^{\text{(rb)}}_J$ for the velocity field. 

For $n=0,\dots,N_T-1$ and each $\by \in \bGamma$ one has
\begin{equation}\label{eq:inf_sup_calc}
\begin{aligned}
    \beta^{\text{(rb)},n}(\by)
    &
    \coloneqq
    \inf_{0\neq \lambda_R \in W^{\text{(rb)}}_R }
    \sup_{0 \neq \bphi_J \in \bV^{\text{(rb)},\text{{stab}}}_{J,K,R}} 
    \frac{
        \dotp{\lambda_R}{ \bphi_J \cdot \widehat{\bm}^{n}_{J}(\by)}_{L^2(D)}
    }{  
        \norm{\lambda_R}_{L^2(D)}
        \norm{\bphi_J}_{H^1(D)^3}
    }
    \\
    &
    \geq
    \widetilde{\beta}^{\text{(rb)},n}(\by)
    -
    \norm{
        \widehat{\bm}^{n}_J(\by)
        -
        \mathsf{P}_{\boldsymbol{X}^{\text{(rb)}}_K}
        \widehat{\bm}_h^{n}(\by)
    }_{H^1(D)^3}
    \\
    &
    \geq
    \widetilde{\beta}^{\text{(rb)},n}(\by)
    -
    	\underbrace{
        \norm{
            \widehat{\bm}^{n}_h(\by)
            -
            \mathsf{P}_{\boldsymbol{X}^{\text{(rb)}}_K}
            \widehat{\bm}_h^{n}(\by)
        }_{H^1(D)^3}}_{(\spadesuit)}
        -
        \underbrace{
        \norm{
            \widehat{\bm}^{n}_J(\by)
            -
            \widehat{\bm}_h^{n}(\by)
        }_{H^1(D)^3}}_{(\clubsuit)},
\end{aligned}
\end{equation}
where $\widehat{\bm}^{n}_J(\by)$ is as in \eqref{line:normalize_mag},
and, for each $\by \in \bGamma$, $\widetilde{\beta}^{\text{(rb)},n}(\by)$ is defined 
at time step $t_n$ as
\begin{equation}\label{eq:inf_sup_constant_stab}
    \widetilde{\beta}^{\text{(rb)},n}(\by)
    \coloneqq
    \inf_{0\neq \lambda_R \in W^{\text{(rb)}}_R }
    \sup_{0 \neq \bphi_J \in \bV^{\text{(rb)},\text{{stab}}}_{J,K,R}} 
    \frac{
        \dotp{\lambda_R}{ \bphi_J \cdot \mathsf{P}_{\boldsymbol{X}^{\text{(rb)}}_K} {\widehat{\bm}}^{n}_h(\by)}_{L^2(D)}
    }{  
        \norm{\lambda_R}_{L^2(D)}
        \norm{\bphi_J}_{H^1(D)^3}
    }.
\end{equation}
For any $\lambda \in L^2(D)$ and
$\bphi \in H^1(D)^3$, one has 
\begin{equation}\label{eq:approx_bilinear_b}
	\dotp{\lambda}{ \bphi \cdot \mathsf{P}_{\boldsymbol{X}^{\text{(rb)}}_K} {\widehat{\bm}}^{n}_h(\by)}_{L^2(D)}
	=
	\sum_{k=1}^K
	c^{n}_k(\by)
	\dotp{\lambda}{ \bphi \cdot \boldsymbol{\eta}^{\text{(rb)}}_k}_{L^2(D)^3},
\end{equation}
where $\left(c^n_k(\by)\right)_{k=1}^K \in \mathbb{R}^K$ correspond to the coefficients of the representation of $\widehat{\bm}_h^{n}(\by)$ in $\boldsymbol{X}^{\text{(rb)}}_K$, for each $\by \in \bGamma$, and $\widehat{\bm}_h^{n}(\by)$ is computed as in Line \ref{line:normalize} of Algorithm \ref{algo:TPScheme}.

Observe that, through the reduced basis approximation of
$\widehat{\bm}_h^{n}(\by)$, one can 
obtain, as presented in \eqref{eq:approx_bilinear_b}, an affine-parametric approximation of the bilinear form 
$(\lambda,\bphi ) \mapsto \dotp{\lambda}{ \bphi \cdot \widehat{\bm}_h^{n}(\by)}_{L^2(D)^3}$,
which is a common assumption in the reduced basis method. 
In this particular case, this construction is possible since the aforementioned bilinear form depends on the parameter $\by \in \bGamma$
only through $\widehat{\bm}_h^{n}(\by)$.
Alternatively, 
representations of this form can be computed using the Empirical Interpolation
Method and its discrete counterpart, see, e.g., 
\cite[Chapter 10]{quarteroni2015reduced}.

Following \cite[Section 2.3]{Gerner2012Certified}, 
which in turn uses {\cite[Lemma 3.1]{rozza2007stability}}, one has that
\begin{equation}
	 \widetilde{\beta}^{\text{(rb)},n}(\by)
	 \geq	
	 \inf_{0\neq \lambda_h\in V_h}
    	\sup_{0 \neq \bphi_h \in \bV_h} 
    \frac{
        \dotp{\lambda_h}{ \bphi_h \cdot \mathsf{P}_{\boldsymbol{X}^{\text{(rb)}}_K} {\widehat{\bm}}^{n}_h(\by)}_{L^2(D)}
    }{  
        \norm{\lambda_h}_{L^2(D)}
        \norm{\bphi_h}_{H^1(D)^3}
    },
\end{equation}
namely, the inf-sup constant $\widetilde{\beta}^{\text{(rb)},n}(\by)$
introduced in \eqref{eq:inf_sup_constant_stab} is bounded from below 
by the corresponding inf-sup constant in the high-fidelity spaces. 
This property follows from the use of the reduced space $\bV^{\text{(rb)},\text{{stab}}}_{J,K,R}$
in the definition of  $\widetilde{\beta}^{\text{(rb)},n}(\by)$. Observe that this property
is valid pointwise for each $\by \in \bGamma$ and each time step $t_n$.

As a consequence of~\cite[Lemma 5.5]{Akrivis2021Higher}\footnote{{Direct application of the result~\cite[Lemma 5.5]{Akrivis2021Higher} requires 
$\snorm{\mathsf{P}_{\boldsymbol{X}^{\text{(rb)}}_K} {\widehat{\bm}}^{n}_h(\by)}=1$
a.e in $D$, which we do not have. However, an inspection of the proof
reveals that, for the $L^2(D)$-case in question, this condition is not necessary.}}, one has that for each
$B>0$ there exists $h_B>0$ such that
provided $\norm{\mathsf{P}_{\boldsymbol{X}^{\text{(rb)}}_K} {\widehat{\bm}}^{n}_h(\by)}_{W^{1,\infty}(D)}\leq B$ one has that for any $h<h_B$ it holds $\widetilde{\beta}^{\text{(rb)},n}(\by)\gtrsim h^{-1}>0$ for each $\by \in \bGamma$.
Furthermore, under the natural assumption that the terms $(\spadesuit)$ and $(\clubsuit)$ decrease as $K$
and $J$ increase, respectively, and considering the last equation of \eqref{eq:inf_sup_calc}, we conclude that
the inf-sup constant $\beta^{\text{(rb)},n}(\by)$, which includes the reduced basis enrichment for the velocity field 
according to \eqref{eq:stabilized_space_velocity}, is strictly positive for sufficiently large $J$ and $K$.

Finally, considering that the bilinear
form $\mathsf{a}(\cdot,\widehat{\bm}^{n}_J(\by),\cdot): H^1(D)^3 \times H^1(D)^3  \rightarrow \IR$, with $\mathsf{a}(\cdot,\cdot,\cdot)$ as in \eqref{eq:trilinear_form_a}, is $H^1(D)^3$-elliptic, one may conclude that the variant of 
Algorithm \ref{algo:POD-TPS-Reduced-Basis} in which $\bV^{\text{(rb)}}_J$ is replaced by the
$\bV^{\text{(rb)},\text{{stab}}}_{J,K,R}$, in particular in the saddle point problem \eqref{eq:saddle_point_for_RB}, is stable for each $\by \in \bGamma$ and at each time step $t_n$.

\section{Surrogate Modelling using Sparse Grid Interpolation}
\label{sec:sg}
Once a reduced basis is available, the Galerkin POD-TPS algorithm from Section~\ref{sec:RB_LLG} provides an efficient reduced order method to approximate the magnetizations $(\bm_h^n(\by))_{n=0}^{N_T}$ for a new parameter $\by\in\bGamma$.
However, two issues arise. 
Firstly, the online reduced problem
as described in Algorithm \ref{algo:POD-TPS-Reduced-Basis} lacks inf-sup stability and stabilizing the velocity basis requires an additional computational effort. Secondly, to approximate $M \gg 1$ new snapshots corresponding to $\by^{\text{(test)}}_1,\dots,\by^{\text{(test)}}_M \in\bGamma$, the online algorithm must be used $M$ times.

To circumvent these issues, we propose an alternative method based on sparse grid interpolation. Interpolation possibly speeds up the computation of many online samples and circumvents any lack of inf-sup stability.

We begin in Section~\ref{sec:sg_params} by recalling the sparse grid interpolation notation and determining the relevant method's parameters. Then, in Section~\ref{sec:sg-rbp}, we define a method based on sparse grid interpolation and the projection of the
parameter-to-solution map on a reduced space.
We refer to~\cite{Bungartz2004Sparse, Nobile2016Convergence}, among others, for a general introduction to sparse grids.

\subsection{Sparse Grid Interpolation Notation and Parameters}\label{sec:sg_params}
In this section, we outline the sparse grid notation and relevant parameters.
In the following, we consider:
A \emph{1d nodes family} $(\YY_{\nu})_{\nu\in\N_0}$, which is a sequence of \emph{nodes tuples}
$\YY_{\nu} = \set{ y_1, \dots, y_{m(\nu)} }\subset\R$, 
i.e. sets of $m(\nu) \in \N$ nodes such that $\YY_0=\set{0}$ and $\#\YY_{\nu} \leq \#\YY_{\nu+1}$ for all $\nu\in\N_0$;
A \emph{1d interpolation method} $(I_{\nu})_{\nu\in\N_0}$, where $I_{\nu}$ interpolates continuous functions of one real variable over the nodes tuple $\YY_{\nu}$;
A $\emph{multi-index set}$ $\Lambda$, i.e.~a finite subset of $\N_0^s$ for $s\in\N$. 
$\Lambda$ is called \emph{downward-closed} if for all $\bnu\in\Lambda$, $\bnu-\be_{n}\in\Lambda$ for all $n=1,\dots,s$ such that $\nu_n>0$.

Based on the \emph{sparse grid construction}~\cite{Bungartz2004Sparse, Nobile2016Convergence} (which we do not recall here), the parameters above define uniquely a \emph{sparse grid} $\HH$, i.e.~a finite set of nodes in $\bGamma$ which generalizes tensor products of one-dimensional node tuples. 
For $\by\in\HH$, we denote by $L_{\by}$ the Lagrange interpolation basis function on $\HH$, which equals $1$ on $\by$ and 0 on $\HH\setminus\set{\by}$.
Finally, the \emph{sparse grid interpolation operator} applied to any continuous function $u$ on $\bGamma$ reads:
\begin{align*}
	\II[u](\by') = \sum_{\by\in\HH} u(\by) L_{\by}(\by'),\qquad \forall\by'\in\bGamma.
\end{align*}

We now recall the precise choice of method's parameters which were found in~\cite[Section~1.2.3, Section~7]{an2025sparse} to give a dimension-independent approximation rate of the pLLG parameter-to-solution map $\by\mapsto \bm(\by)$: 
The 1D nodes family is chosen, for any $\nu>0$, as
\begin{align*}
	\YY_{\nu} = \set{y_i \coloneqq \phi \left( -1+\frac{2i}{\left(2^{\nu+1}-1\right)+1} \right) \text{ for } i = 1, \ldots, m(\nu) = 2^{\nu+1}-1},
\end{align*}
where $\phi(x) \coloneqq 2 \sqrt{2} {\erf}^{-1}(x)$ for $x\in(-1,1)$ and $\erf(x) = \frac{2}{\sqrt{\pi}}\int_0^x e^{-t^2} \rmd t$ denotes the \emph{error function}. Intuitively, $\phi: (-1, 1) \rightarrow \R$ transforms uniformly distributed nodes into standard normally distributed ones. Using $m(\nu) = 2^{\nu+1}-1$ nodes guarantees that the nodes family is nested as $\nu\in\N_0$ increases.
The {1D interpolation operator} $I_{\nu}: C^0(\R)\rightarrow C^0(\R)$ is defined as the constant $I_0 f \equiv f(0)$, if $\nu=0$, and a continuous piecewise-polynomial interpolation if $\nu \geq 1$. In particular, it consists of polynomial interpolation of degree $p\geq 1$ in the intervals defined by the 1d nodes, and extension of the nearby polynomial otherwise.
The multi-index set $\Lambda$ is defined using a \emph{profit function} $\PP: \N_0^s \rightarrow \R_+$.
In particular, $\Lambda$ must contain all multi-indices $\bnu\in \N_0^s$ with profit above a given threshold $\eps>0$, i.e. 
$\Lambda = \Lambda(\eps) = \setc{\bnu\in\N_0^s}{\PP_{\bnu} > \eps}$.
The profit function reads:
\begin{align*}
    \PP_{\bnu} = \prod_{i : \nu_i=1} 2^{-\frac{3}{2}\lceil\log_2(i)\rceil} \prod_{i : \nu_i>1} \left( 2^{\nu_i + \frac{1}{2}\lceil\log_2(i)\rceil} \right)^{-p} \left( \prod_{i:\nu_i\geq1} p2^{\nu_i} \right)^{-1}.
\end{align*}
This profit function was defined in~\cite{an2025sparse} based on a regularity analysis of the parameter-to-solution map and guarantees that  $\Lambda$ is downward-closed for any value of the profit threshold.

\subsection{Sparse Grid-Reduced Basis Projection (SG-RBP)}\label{sec:sg-rbp}
In this section, we describe an alternative surrogate modeling approach for the pLLG equation based on sparse grid interpolation of reduced basis coefficients. 
This method leverages the reduced basis constructed in the offline phase and uses sparse grid interpolation with parameters defined in Section~\ref{sec:sg_params} to efficiently approximate the parametric dependence of the magnetization field.

\begin{algorithm}[H]
	\caption{SG-RBP: Sparse grid interpolation of the reduced basis coefficients}
	\label{algo:SG-RBP}
	\hspace*{\algorithmicindent} {\bf Input:} $\HH$, $N_T$, $\bX^{\text{(rb)}}_K$, $\by'$, $(L_{\by})_{\by\in\HH}$ \\
	\hspace*{\algorithmicindent} {\bf Output:} $\left(\II[\bm_K^n](\by')\right)_{n=0}^{N_T}$
	\begin{algorithmic}[1] 
		\For{$\by \in \HH$}
			\State Compute $\left(\bm_h^n(\by)\right)_{n=0}^{N_T}$ with the high fidelity tangent plane scheme (Algorithm~\ref{algo:TPScheme})
		\EndFor
		\For{$n = 1, \dots, N_T$ and $\by\in\HH$}
			\State Compute $\bm_K^n(\by) = \mathsf{P}^{\text{(rb)}}_K \bm_h^n(\by)$
		\EndFor
		\For{$n = 1, \dots, N_T$}
			\State Compute $\II[\bm_K^n](\by') = \displaystyle\sum_{\by\in\HH} \bm_K^n(\by) L_{\by}(\by')$
		\EndFor
	\end{algorithmic}
\end{algorithm}

This method presents some advantages over the Galerkin POD-TPS scheme from Section~\ref{sec:RB_LLG}:
It avoids online Galerkin computations, thus circumventing possible lacks of stability; and, compared to sparse grid interpolation of the high fidelity samples (without reduced basis projection), the dimensionality of the interpolated coordinates vectors is smaller, which results in an enormous speedup.

The convergence of the sparse grid interpolation of the \emph{high-fidelity} LLG samples was studied in~\cite{an2025sparse}. The result~\cite[Theorem~7.8]{an2025sparse} carries over, as shown in the next proposition.
\begin{proposition}
    Consider the parameter-to-solution map $\by\mapsto \bm(\by)$ arising from the pLLG equation~\eqref{eq:pLLG} but assuming that 
    $\bH_{\text{eff}}(W, \bm) = \widehat{\CC}(W, \bm) = W \bm \times \Delta \bg$.
    Additionally, assume that the initial condition is Hölder continuous, i.e. $\bm^0 \in C^{2+\delta}(D)^3$ for some $0<\delta<1$ with sufficiently small $C^{2+\delta}(D)^3$-seminorm.
    Denote by $\mathsf{P}^{\normalfont\text{(rb)}}_K$ the orthogonal projection operator onto the magnetization reduced space, as described in Section~\ref{sec:RB_LLG}.
    Denote by $\II$ the sparse grid interpolation operator defined from the parameters as in Section~\ref{sec:sg_params} and by $\HH\subset \bGamma$ the corresponding sparse grid.
    
    Then, for any $\frac{2}{3} < \tau < 1$, there exists $C=C(\tau, p)>0$ such that
    \begin{align}
        \norm{\bm_{K} - \II[\bm_K]}_{L^2_{\mu} (\bGamma, L^2(I, H^1(D)^3))}
        \leq C \left(\#\HH\right)^{1-1/\tau}.
    \end{align}
    
\end{proposition}
\begin{proof}
    In~\cite[Theorem~7.8]{an2025sparse}, the convergence of sparse grid interpolation is proved assuming the regularity of the parameter-to-solution map $\by\mapsto \bm(\by)$. 
    More precisely, they assume bounds on the mixed derivatives of the parameter-to-solution map like the ones proved in~\cite[Proposition 6.5]{an2025sparse}. Since here we interpolate the reduced basis projection $\bm^n_K(\by) = \mathsf{P}^{\text{(rb)}}_K \bm^n_h(\by)$, a \emph{linear} function of $\bm^n_h(\by)$ for all $n=1, \dots, N_T$, the same estimates hold.
    As a consequence, also the same sparse grid interpolation convergence rates hold.
\end{proof}

The previous proposition requires strong assumption on the pLLG problem we are solving. In particular, no external magnetic field can be included in the effective field. Moreover, the simplified expression of $\widehat{\CC}(W, \bm)$ can only be a good approximation of Equation~\eqref{eq:pLLG} in the short time and uniform noise regime, i.e., $T \ll 1$ and $\nabla \bg\approx \boldsymbol{0}$.

It should be noted that the constant $C=C(p, \tau)$ appearing in the convergence estimate depends on problem parameters. The algebraic convergence with respect to the number of sparse grid nodes, with rate at best $\frac{1}{2}$ for $\tau=\frac{2}{3}$, may be visible only for large sparse grids in the asymptotic regime.

\section{Numerical Results}
\label{sec:numerical_results}
In this section, we test the algorithms introduced in Sections~\ref{sec:RB_LLG} and~\ref{sec:sg}. 
The Python implementation of both methods is based on
DOLFINx~\cite{Baratta2023DOLFINx} for the finite elements computations and 
SGMethods~\cite[Chapter 4]{scaglioni2024thesis} for sparse grid interpolation.
The rest, including the POD-TPS method, is implemented from scratch based on standard libraries such as NumPy and SciPy. 

The present section is structured as follows:
In Section~\ref{sec:setup}, we describe the general setup for the numerical experiments presented in the following two sections.
In Section~\ref{sec:numerics_1d}, we consider a test problem with a one-dimensional parameter space.
In Section~\ref{sec:numerics_nd}, we consider high parametric dimensions.
In Section~\ref{sec:numerics_switching}, we test the algorithms on a switching dynamics, which is a less academic example than those considered in the previous sections.

\subsection{Setup}\label{sec:setup}

In Sections~\ref{sec:numerics_1d} and~\ref{sec:numerics_nd} below, we consider examples of the parametric LLG equation (see Section~\ref{sec:model_LLG}) based on the following common 
setting: The physical domain $D$ is given by
\begin{equation}\label{eq:domain_D}
    D 
    \coloneqq
    \left\{
        \bx\in\R^3: 0\leq x_1 \leq 1, 0\leq x_2 \leq 1,\ x_3=0
    \right\}
\end{equation}
and the time interval is $I=[0, T]$ with $T=0.5$.
As an initial condition, we consider 
\begin{equation}
    \bm^0 
    \coloneqq
    \left(0.9 \left(x_1-\frac{1}{2}\right), 0.9 \left(x_2-\frac{1}{2}\right), \sqrt{1-(0.9r)^2}\right)^\top,
\end{equation}
where $r^2 = \left(x_1-\frac{1}{2}\right)^2+\left(x_2-\frac{1}{2}\right)^2$,
and set the Gilbert damping to $\alpha = 1.4$. In addition,
we consider the following spatial noise distribution 
\begin{equation}
    \bg 
    \coloneqq
    \left(
        0.5\sin\left(\pi \left(x_1-\frac{1}{2}\right)\right), 
		0.5\sin\left(\pi \left(x_2-\frac{1}{2}\right)\right), 
		\sqrt{1- g_0^2 - g_1^2 }
	\right)^\top. 
\end{equation}
The external magnetic field is set to zero, $\bH_{\text{ext}} \equiv \mathbf{0}$.

The space and time discretization parameters change across the different sets of numerical results. 
Unless explicitly stated, we use:
Piecewise linear finite elements on structured triangular meshes (see Figure~\ref{fig:mesh}, left) consisting of 128 elements and a mesh size $h=0.125$, order 1 BDF time stepping, and time step size $\tau = 10^{-3}$.
See Figure~\ref{fig:mesh}, center, for a plot of $\bm^0$, and right, for a sample magnetization at final time $t=T$.
\begin{figure}
	\centering
	\includegraphics[width=0.33\textwidth,trim={10 10 10 10},clip]{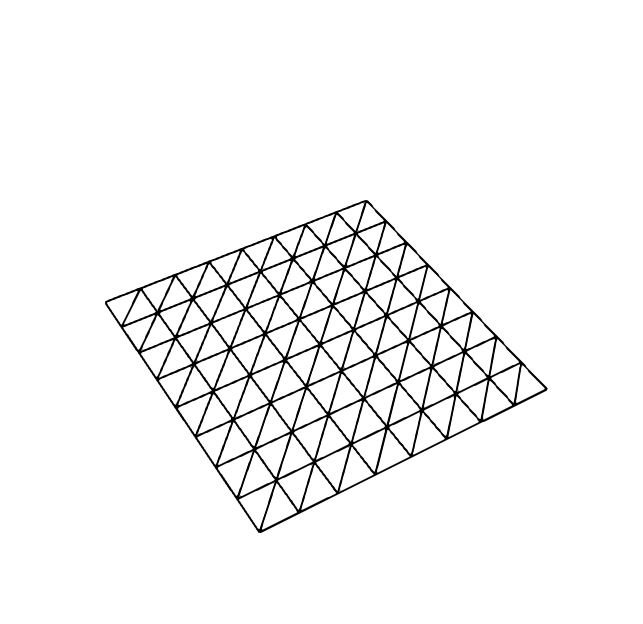}\hspace{-0.2cm}
	\includegraphics[width=0.33\textwidth,trim={10 10 10 10},clip]{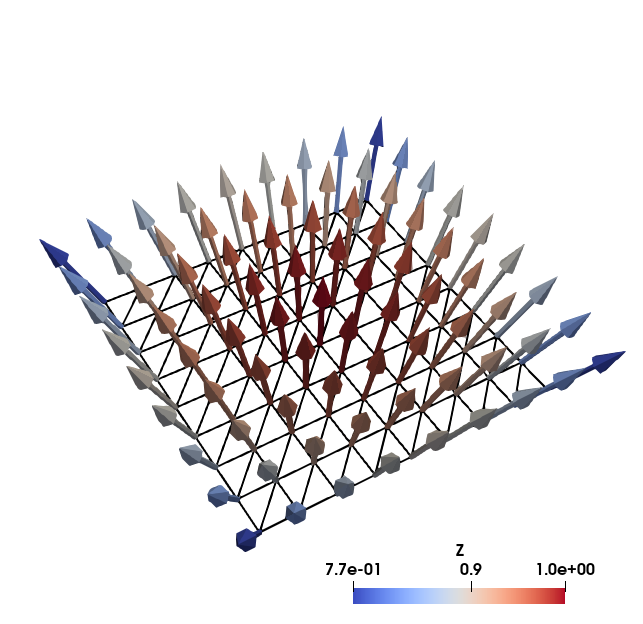}\hspace{-0.2cm}
	\includegraphics[width=0.33\textwidth,trim={10 10 10 10},clip]{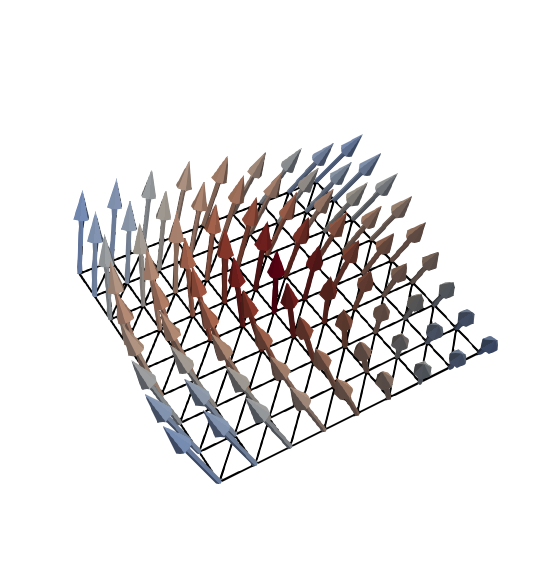}
	\caption{Left: The domain (2D unit square in the $xy$ plane) and an example of mesh used for the numerical tests.
	Center and Right: A snapshot of the random magnetization respectively at times $t=0$ and $t=1$ and colored with the $z$ component.}
	\label{fig:mesh}
\end{figure}

Throughout this section, we consider $N_S = 128$ samples $\set{\by^{(1)}, \dots, \by^{(N_S)}} \subset \bGamma$ in the parameter distributed according to a standard Normal distribution. The high-fidelity samples are computed using the TPS as introduced in Section~\ref{sec:tangent_plane_scheme}, in particular using the saddle point formulation stated in \eqref{eq:saddle_point_for}.
This results in $N_S$ velocities, Lagrange multipliers, and magnetizations sample paths, which in turn are understood as a sequence of elements in the finite element space over the time steps. 
Following Section~\ref{eq:empirical_POD}, we construct
reduced spaces $\bV^{\text{(rb)}}_{J}$, $\boldsymbol{X}^{\text{(rb)}}_{K}$ and $W^{\text{(rb)}}_R$ (of dimensions $J$, $K$, and $R$)
for the velocity, magnetization, and the Lagrange multiplier, respectively 
(cf. Remark \ref{rmk:POD_other}).

In order to assess the accuracy of the RB method and the variants introduced in this work, we consider the following error metrics.

\begin{itemize}
\item[(i)] {\bf Projection Error.}
We consider a collection of test points
$\left\{\by^{\text{(test)}}_1,\cdots,\by^{\text{(test)}}_{M}\right\}$
and define
\begin{equation}\label{eq:error_metric}
    \eps^{\text{Proj}}_{\bm,K}
    =
    \sqrt{
    \frac{1}{M N_T}
    \sum_{m=1}^{M}
    \sum_{n=1}^{N_T} 
    \norm{
        \bm^n_h\left(\by^{\text{(test)}}_m\right)
        - 
        \mathsf{P}_{K} \bm^n_h\left(\by^{\text{(test)}}_m\right)
    }^2_{H^1(D)^3}
    },
\end{equation}
where $\mathsf{P}_{K}: H^1(D)^3 \rightarrow \boldsymbol{X}^{\text{(rb)}}_{K}$ is the $H^1(D)^3$-based projection into $\boldsymbol{X}^{\text{(rb)}}_{K}$.
Observe that \eqref{eq:error_metric} approximates the $L^2_{\mu}(\bGamma, L^2(I, H^1(D)^3))$ error.
Equivalently, we can define $\eps_{\bv,R}$ and $\eps_{\lambda,R}$ for the
Lagrange multiplier on the corresponding reduced spaces. 
\item[(ii)] {\bf Galerkin POD Error.}
We consider the Galerkin POD error, defined for the magnetization field as 
\begin{equation}
    \eps^{\text{G-POD}}_{\bm,J}
    =
    \sqrt{
    \frac{1}{M N_T}
    \sum_{m=1}^{M}
    \sum_{n=1}^{N_T} 
    \norm{
        \bm^n_h\left(\by^{\text{(test)}}_m\right)
        - 
        \bm^n_J \left(\by^{\text{(test)}}_m\right)
    }^2_{H^1(D)^3}
    },
\end{equation}
where $\bm^n_J(\by) \in \bV^{\text{(rb)}}_{J}$ is
the output of the Galerkin POD-TPS stated in Algorithm \ref{algo:POD-TPS-Reduced-Basis}, in particular computed as in \eqref{eq:update_m_RB_TPS}.
If in Algorithm \ref{algo:POD-TPS-Reduced-Basis} we use the stabilized reduced space
$\bV^{\text{(rb)},\text{{stab}}}_J$ introduced in \eqref{eq:stabilized_space_velocity},
then we adopt the notation $\eps^{\text{G-POD,Stab}}_{\bm,J,K,R}$ for the Galerkin POD error,
which reflects the fact that the dimension of this space not only depends on $J$ but also on 
the dimension of magnetization and Lagrange multiplier reduced spaces $K$ and $R$.
\end{itemize}

In the following and for the numerical results
presented herein, we use $M=30$ and sample each component of the $M$ test points as i.i.d. standard Normal random variables in each component.

\subsection{1-dimensional Parameter Space} \label{sec:numerics_1d}
We consider first a simplified example in which the parameter space is $\bGamma = \R$, which 
is equivalent to considering a single parametric dimension.
The aim of this section is to demonstrate the performance of the reduced basis approach in a simple setting and to validate the implementation.

\subsubsection{Offline Phase}
Together with the error metrics introduced in Section \ref{sec:setup},
we assess the performance of the reduced basis by inspecting the decay of the singular values of the POD for each of these quantities. 
Figure~\ref{fig:svs_proj} portrays the singular values 
obtained in the construction of the reduced spaces
and projection errors. The decay of the singular values and of the projection error suggests a good quality of the reduced basis and its suitability for the Galerkin approximation in the online phase.
Figure~\ref{fig:rbs} displays the first three reduced basis functions for the magnetization field.

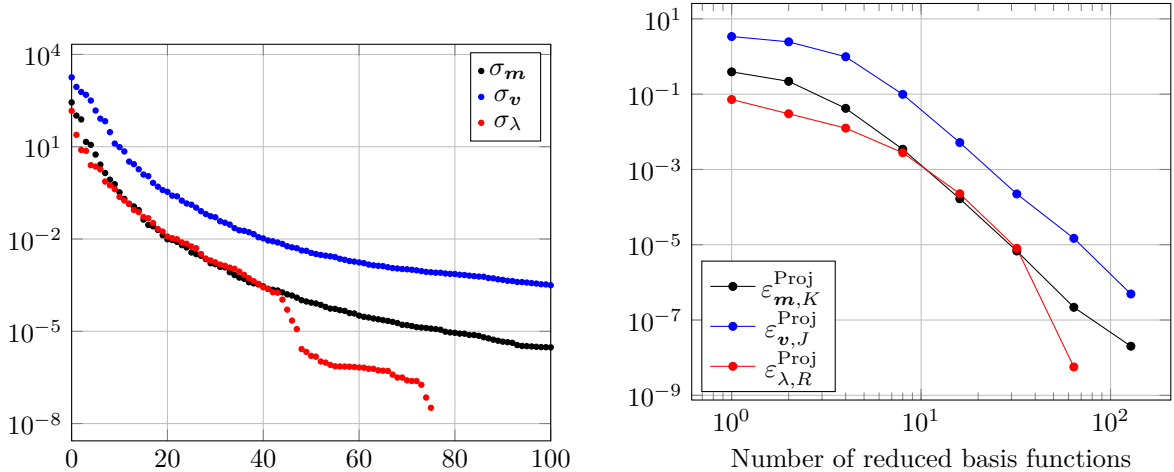
\begin{figure}
\center
    \begin{tikzpicture}
        \begin{semilogyaxis}[
            xmin=0, xmax=100,
            grid=major,
            width=0.51\linewidth
        ]
            \addplot[only marks, mark=*, mark size=1pt] table[col sep=comma, x expr=\coordindex, y index={0}]{data_plots/1d/singular_values_m.csv};
            \addlegendentry{$\sigma_{\bm}$}
            \addplot[only marks, mark=*, mark size=1pt, blue] table[col sep=comma, x expr=\coordindex, y index={0}]{data_plots/1d/singular_values_v.csv};
            \addlegendentry{$\sigma_{\bv}$}
            \addplot[only marks, mark=*, mark size=1pt, red] table[col sep=comma, x expr=\coordindex, y index={0}]{data_plots/1d/singular_values_l.csv};
            \addlegendentry{$\sigma_{\lambda}$}
        \end{semilogyaxis}
    \end{tikzpicture}
    \hspace{0.02\textwidth}
    \begin{tikzpicture}
        \begin{loglogaxis}[
            grid=major,
            xlabel={Number of reduced basis functions},
            width=0.51\linewidth
            , legend style={at={(0.02,0.02)}, anchor=south west}
        ]
            \addplot[ mark=*, mark size=1.5pt] table[col sep=comma, x index={0}, y index={1}]{data_plots/1d/conv_proj_m.csv};
            \addlegendentry{$\eps^{\text{Proj}}_{\bm,K}$}
            \addplot[mark=*, mark size=1.5pt, blue] table[col sep=comma, x index={0}, y index={1}]{data_plots/1d/conv_proj_v.csv};
            \addlegendentry{$\eps^{\text{Proj}}_{\bv,J}$}
            \addplot[mark=*, mark size=1.5pt, red] table[col sep=comma, x index={0}, y index={1}]{data_plots/1d/conv_proj_l.csv};
            \addlegendentry{$\eps^{\text{Proj}}_{\lambda,R}$}
        \end{loglogaxis}
    \end{tikzpicture}
    \caption{Left: First 100 singular values of velocities ($\sigma_{\bv}$), Lagrange multipliers ($\sigma_{\lambda}$), and magnetizations ($\sigma_{\bm}$) based on $N_S = 128$ snapshots. Right: Projection errors onto the reduced spaces for velocities, Lagrange multipliers, and magnetizations with respect to the number of reduced basis functions.}
    \label{fig:svs_proj}
\end{figure}

\begin{figure}
	\centering
	\includegraphics[width=0.32\textwidth,trim=40 20 40 20,clip]{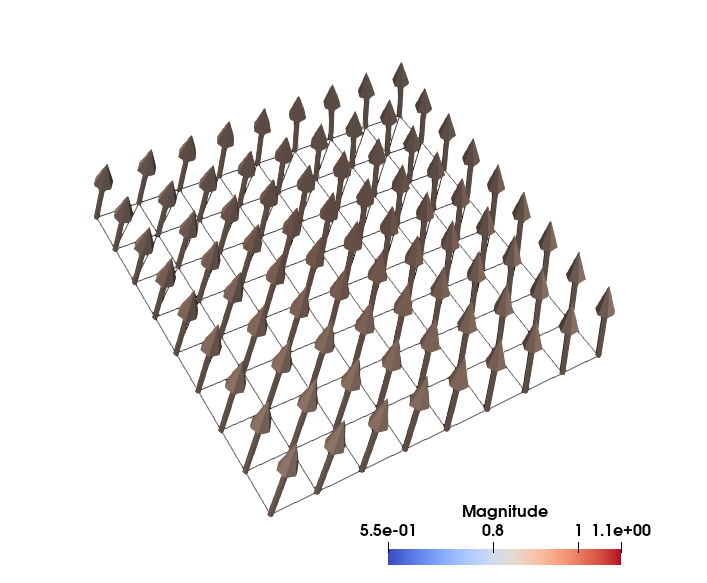}
	\hspace{-0.5em}
	\includegraphics[width=0.32\textwidth,trim=40 20 40 20,clip]{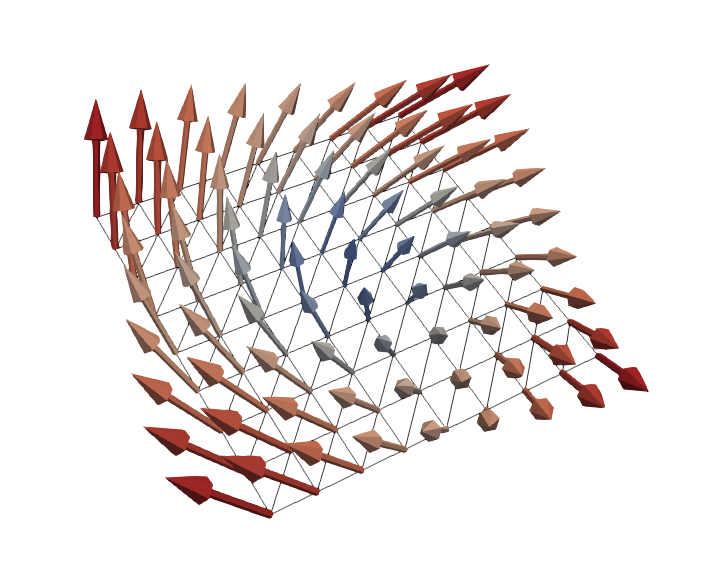}
	\hspace{-0.5em}
	\includegraphics[width=0.32\textwidth,trim=10 10 10 20,clip]{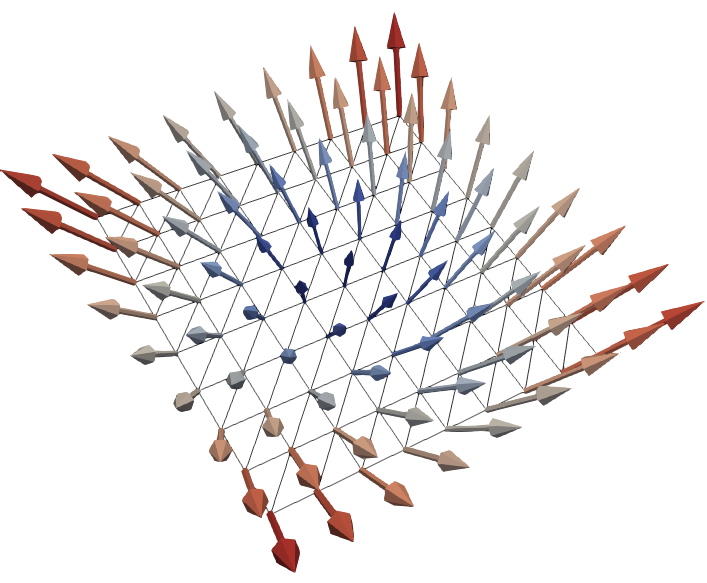}
	\caption{First three magnetization reduced basis functions. The three vector fields are color-coded with the magnitude. Each reduced basis function is defined only up to a constant. Therefore, they are scaled for ease of readability.}
	\label{fig:rbs}
\end{figure}

\subsubsection{Online Phase}
Next, we present a comprehensive set of 
numerical experiments addressing the performance of the Galerkin POD-TPS method in the online phase.

\subsubsection*{Number of Reduced Basis Functions and Supremizer Stabilization}
We numerically investigate the convergence of the reduced basis approximation 
of the magnetization field with respect to the total number of reduced bases used for velocities and Lagrange multipliers.
More precisely, we compare the performance of the following four choices:
\begin{itemize}
	\item \textbf{Online Galerkin-Equal (OG-1x).} 
    We consider the computation of the magnetization field in the online phase 
    using the Galerkin POD-TPS as introduced in Algorithm \ref{algo:POD-TPS-Reduced-Basis} while employing the same number of reduced basis functions for velocities and Lagrange multipliers, i.e. $\text{dim}(\bV_J^{\text{(rb)}}) = \text{dim}(W_R^{\text{(rb)}})$, therefore $J = R$ in this case.
	\item \textbf{Online Galerkin-3x factor (OG-3x).}
    As in the previous case, we consider the computation of the magnetization field in the online phase using the Galerkin POD-TPS as introduced in Algorithm \ref{algo:POD-TPS-Reduced-Basis} while using three times as many reduced basis functions for velocities as for Lagrange multipliers, i.e. $\text{dim}(\bV_J^{\text{(rb)}}) = 3 \text{dim}(W_R^{\text{(rb)}})$, therefore $J = 3 R$ in this case. 
	\item \textbf{Supremizer-stabilized OG-1x (SS-OG-1x)}: 
    We consider the computation of the magnetization field in the online phase 
    according to Algorithm \ref{algo:POD-TPS-Reduced-Basis}. However, unlike in the two 
    previous approaches, we use $\bV^{\text{(rb)},\text{{stab}}}_{J,K,R}$ as in 
    \eqref{eq:stabilized_space_velocity} for the approximation of the velocity field,
    which in this case is enriched in total with
    at most $J$ supremizers, with $K = \lfloor \sqrt{J}\rfloor $ and
    $R = \lfloor \sqrt{J}\rfloor $.
	\item \textbf{Supremizer-stabilized OG-3x (SS-OG-3x).}
    Again, we use $\bV^{\text{(rb)},\text{{stab}}}_{J,K,R}$ as in 
    \eqref{eq:stabilized_space_velocity} instead of $\bV^{\text{(rb)}}_{J}$ for the approximation of the velocity field. However, now we enriched 
    reduced space for the velocity field in total with at most $J$ supremizers.
    These are distributed following
    $K = \lfloor\frac{ \sqrt{J}}{3}\rfloor $ and
    $R = 3K$.
\end{itemize}
Figure~\ref{fig:err_inf_sup_online_galerkin} presents the results for these four methods. In particular, we present the convergence of the Galerkin POD error (left) and the smallest 
inf-sup constant of the saddle point problem across all time steps. 
Concerning these results, we make the following observations:
\textbf{OG-1x} is clearly not inf-sup stable and the error does not diminish as the number of reduced basis functions increases. 
\textbf{OG-3x} has an inf-sup constant that is closer to the one of the high-fidelity model and displays better convergence. 
The addition of the supremizer stabilization always improves inf-sup stability, in particular for \textbf{SS-OG-1x}, which now also converges.
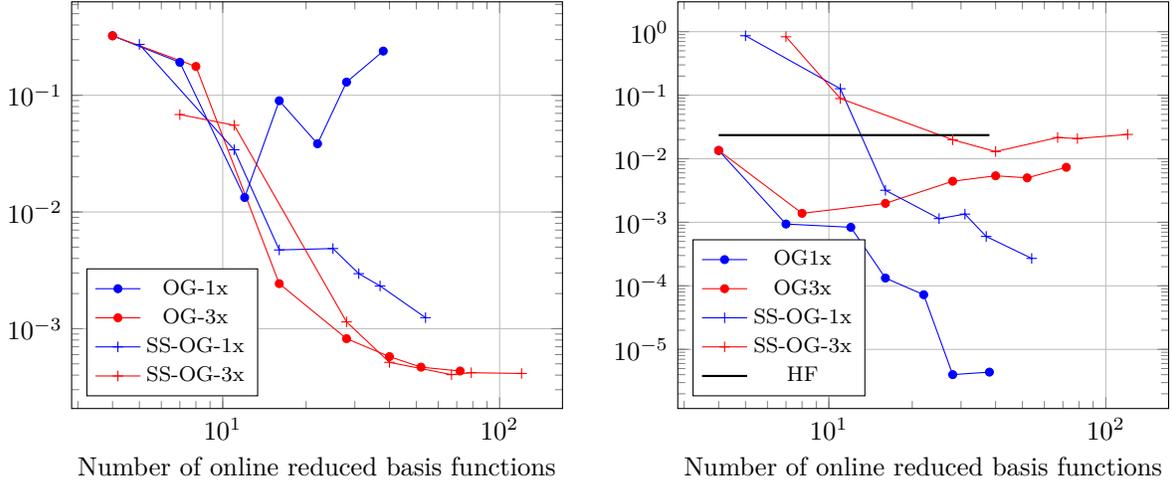
\begin{figure}
	\center
    \begin{tikzpicture}
        \begin{loglogaxis}[
            width=0.52\linewidth,
            grid=major,
            xlabel={Number of online reduced basis functions},
            legend style={ font=\footnotesize}, 
            legend pos = south west
        ]
            \addplot[mark=*, mark size=1.5pt, blue] table[col sep=comma, x expr=\thisrowno{1}+\thisrowno{2}, y index={4}
            ]{data_plots/1d/conv_test_OG1x.csv};
            \addlegendentry{OG-1x}
            \addplot[mark=*, mark size=1.5pt, red] table[col sep=comma, x expr=\thisrowno{1}+\thisrowno{2}, y index={4}
            ]{data_plots/1d/conv_test_OG3x.csv};
            \addlegendentry{OG-3x}
            \addplot[mark=+, mark size=2pt, blue] table[col sep=comma, x expr=\thisrowno{1}+\thisrowno{3}, y index={4}
            ]{data_plots/1d/conv_test_SSOG1x.csv};
            \addlegendentry{SS-OG-1x}
            \addplot[mark=+, mark size=2pt, red] table[col sep=comma, x expr=\thisrowno{1}+\thisrowno{3}, y index={4}
            ]{data_plots/1d/conv_test_SSOG3x.csv};
            \addlegendentry{SS-OG-3x}
        \end{loglogaxis}
    \end{tikzpicture}
    \hspace{0.02\textwidth}
    \begin{tikzpicture}
        \begin{loglogaxis}[
            grid=major,
            xlabel={Number of online reduced basis functions},
            legend style={font=\footnotesize}, 
            legend pos = south west,
            width=0.52\linewidth
        ] 
            \addplot[mark=*, mark size=1.5pt, blue] table[col sep=comma, x expr=\thisrowno{1}+\thisrowno{2}, y index={5}
            ]{data_plots/1d/conv_test_OG1x.csv};
            \addlegendentry{OG1x}
            \addplot[mark=*, mark size=1.5pt, red] table[col sep=comma, x expr=\thisrowno{1}+\thisrowno{2}, y index={5}
            ]{data_plots/1d/conv_test_OG3x.csv};
            \addlegendentry{OG3x}
            \addplot[mark=+, mark size=2pt, blue] table[col sep=comma, x expr=\thisrowno{1}+\thisrowno{3}, y index={5}
            ]{data_plots/1d/conv_test_SSOG1x.csv};
            \addlegendentry{SS-OG-1x}
            \addplot[mark=+, mark size=2pt, red] table[col sep=comma, x expr=\thisrowno{1}+\thisrowno{3}, y index={5}
            ]{data_plots/1d/conv_test_SSOG3x.csv};
            \addlegendentry{SS-OG-3x}
            \addplot[thick] table[col sep=comma, x expr=\thisrowno{1}+\thisrowno{2}, y index={6}
            ]{data_plots/1d/conv_test_OG1x.csv};
            \addlegendentry{HF}
        \end{loglogaxis}
    \end{tikzpicture}
	\caption{Comparison of the four online Galerkin POD-TPS strategies: 
		\textbf{OG-1x}, 
		\textbf{OG-3x},
		and the respective versions with supremizer stabilization, \textbf{SS-OG-1x} and \textbf{SS-OG-3x}.
		Left: Galerkin POD errors $\eps^{\text{G-POD}}_{\bm,J}$ and $\eps^{\text{G-POD,Stab}}_{\bm,J,K,R}$ of the magnetization field as a function of the total number of reduced basis functions.
		Right: Minimum inf-sup constant over time-steps as a function of the total number of reduced basis functions. In black, the minimum inf-sup constant over the time steps of the high fidelity solver.}
	\label{fig:err_inf_sup_online_galerkin}
\end{figure}

The stagnation of the error displayed by the convergent methods is a consequence of the constant time step size, as we show next.

\subsubsection*{Convergence Under Online Time Step Refinement}

Our conjecture is that the lack of convergence of \textbf{OG-3x}, \textbf{SS-OG-1x} and \textbf{SS-OG-3x} in the previous experiments is due to the time step size used in the online phase, through the consistency errors of the BDF extrapolation and time derivative approximation.

To test this, we again use the same reduced bases for velocities, Lagrange multipliers, and magnetizations as in the previous experiment (computed with a relatively large time step).
However, we perform the online phase with different time step sizes.

Figure~\ref{fig:err_infsup_convergence_dt} shows, for each value of the time step $\tau$ used in the online phase, the error of the reduced model with 
reduced basis of a dimension given according to the Remark \ref{eq:criterium}  with $\epsilon_{\normalfont\text{POD}}^2 = 10^{-5}$ also guaranteeing that $\text{dim}(\bV_J^{\text{(rb)}}) \geq 3 \text{dim}(W_R^{\text{(rb)}})$. This results in $\text{dim}(\bV_J^{\text{(rb)}})=30$ and $\text{dim}(W_R^{\text{(rb)}})=10$.
We use a reference solution computed with the high-fidelity model with smaller time steps than those considered in the online phase.
The error converges with a rate close to 1 while the inf-sup constant is independent of the time step size. Similar results hold when the supremizer stabilization is employed.
\begin{figure}
	\center
    \begin{tikzpicture}
        \begin{loglogaxis}[
            grid=major,
            xlabel={Time step size $\tau$},
            legend style={at={(0.97,0.03)}, anchor=south east},
            width=0.49\linewidth
        ]
            \addplot[mark=*, mark size=1.5pt, blue] table[col sep=comma, x index={0}, y index={3}
            ]{data_plots/1d/conv_test_dt.csv};
            \addlegendentry{OG-3x}
            \addplot[black, thick, domain=3e-3:1.e-1, samples=100] {1.2*x};
            \addlegendentry{$\tau$}
        \end{loglogaxis}
    \end{tikzpicture}
    \hspace{0.02\textwidth}
    \begin{tikzpicture}
        \begin{loglogaxis}[
            grid=major,
            xlabel={Time step size $\tau$},
            legend style={at={(0.97,0.03)}, anchor=south east},
            ymin=1.e-3, ymax=1.e-1,
            width=0.49\linewidth
            ]
            \addplot[mark=*, mark size=1.5pt, blue] table[col sep=comma, x index={0}, y index={4}
            ]{data_plots/1d/conv_test_dt.csv};
            \addlegendentry{OG-3x}
            \addplot[thick] table[col sep=comma, x index={0}, y index={5}
            ]{data_plots/1d/conv_test_dt.csv};
            \addlegendentry{HF}
        \end{loglogaxis}
    \end{tikzpicture}
	\caption{Study of \textbf{OG-3x} for different time step sizes (used only in online phase).
	Left: Error as a function of the time step size.
	Right: Minimum inf-sup constant over time-steps as a function of the time step size.}
	\label{fig:err_infsup_convergence_dt}
\end{figure}
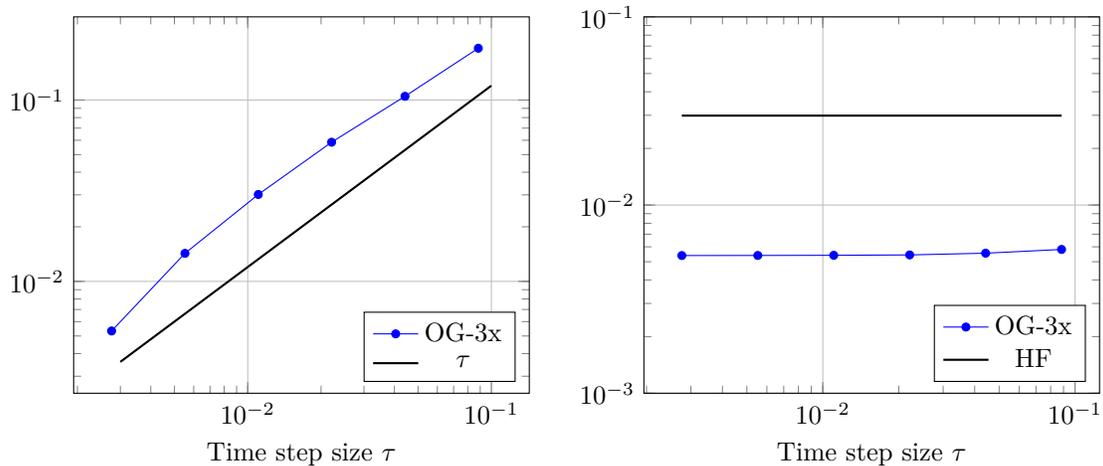

\subsubsection*{Convergence Under Snapshots Mesh Refinement}
We study the effect of the mesh size (used to compute the snapshots) on the performance of the Galerkin POD-TPS method described in Algorithm \ref{algo:POD-TPS-Reduced-Basis}.
In particular, we study the effect of the mesh size on the Galerkin POD error and on the inf-sup stability of the method at each time step.

For the high-fidelity model, $h$-refinement leads to order 1 convergence of the $L^2_{\mu}(\bGamma, L^2(I, H^1(D)^3))$ error and an order 1 decrease of the minimum inf-sup constant of the TPS linear systems.
We confirm this with a convergence test on 4 meshes on the unit square (as in Section~\ref{sec:setup}) with decreasing mesh size 
$h= \left\{2^{-2}, \ldots, 2^{-5}\right\}$.
Results are displayed in Figure~\ref{fig:err_infsup_h}.

We then test the Galerkin POD-TPS algorithm as follows: We again fix the time step size to $\tau=10^{-3}$ and
for each of the meshes of size $2^{-2}$ to $2^{-5}$, we carry out the offline phase (snapshots sampling and POD) and convergence tests for the \textbf{OG-3x} and \textbf{SS-OG-3x} algorithms. We compute the dimension of the reduced space by following Remark \ref{eq:criterium}  with $\epsilon_{\normalfont\text{POD}}^2  = 10^{-5}$.

This results in the number of reduced basis functions and supremizers listed in the following table:
\begin{center}
    \begin{tabular}{l|c|c|c}
        $h$ & $\text{dim}(W_R^{\text{(rb)}})$ & $\text{dim}(\bV_J^{\text{(rb)}})$  & $\text{dim}\left(\bV_{J,K,R}^{\text{(rb), stab}}\right)$ \\
        \hline
        $2^{-2}$ & 9 & 27 & 54\\
        $2^{-3}$ & 10 & 30 & 57\\
        $2^{-4}$ & 12 & 36 & 63\\
        $2^{-5}$ & 10 & 30 & 57
    \end{tabular}
\end{center}
The mesh of size $2^{-6}$ is reserved for the reference solution, computed on $M$ new test parameters (different from the ones used for the snapshots).
The approximation error consists in this case of two contributions: 
One from the finite element discretization (the snapshots are coarser than the reference), the other from the reduced basis approximation.
See Figure~\ref{fig:err_infsup_h}.
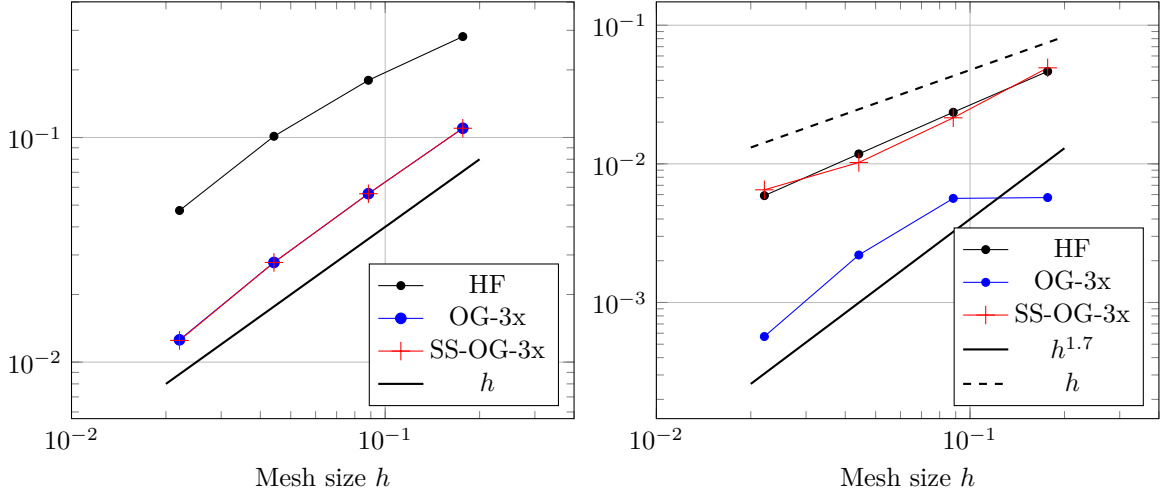
\begin{figure}
	\center
    \begin{tikzpicture}
        \begin{loglogaxis}[
                grid=major,
                xlabel={Mesh size $h$},
                legend style={at={(0.97,0.03)}, anchor=south east},
                xmin=1.e-2, xmax=4.e-1,
                width=0.53\linewidth
            ]
            \addplot[mark=*, mark size=1.5pt, black] table[col sep=comma, x index={0}, y index={2}
            ]{data_plots/1d/conv_test_HF_vs_h.csv};
            \addlegendentry{HF}
            \addplot[mark=*, mark size=2pt, blue] table[col sep=comma, x index={0}, y index={5}
            ]{data_plots/1d/conv_test_OG3x_vs_h.csv};
            \addlegendentry{OG-3x}
            \addplot[mark=+, mark size=3.5pt, red] table[col sep=comma, x index={0}, y index={5}
            ]{data_plots/1d/conv_test_SSOG3x_vs_h.csv};
            \addlegendentry{SS-OG-3x}
            \addplot[black, thick, domain=2e-2:2e-1, samples=100] {0.4*x};
            \addlegendentry{$h$}
        \end{loglogaxis}
    \end{tikzpicture}
    \begin{tikzpicture}
        \begin{loglogaxis}[
            grid=major,
            xlabel={Mesh size $h$},
            legend style={at={(0.97,0.03)}, anchor=south east},
            xmin=1.e-2, xmax=4.e-1,
            width=0.53\linewidth,
        ]
            \addplot[mark=*, mark size=1.5pt, black] table[col sep=comma, x index={0}, y index={3}
            ]{data_plots/1d/conv_test_HF_vs_h.csv};
            \addlegendentry{HF}
            \addplot[mark=*, mark size=1.5pt, blue] table[col sep=comma, x index={0}, y index={6}
            ]{data_plots/1d/conv_test_OG3x_vs_h.csv};
            \addlegendentry{OG-3x}
            \addplot[mark=+, mark size=3.5pt, red] table[col sep=comma, x index={0}, y index={6}
            ]{data_plots/1d/conv_test_SSOG3x_vs_h.csv};
            \addlegendentry{SS-OG-3x}
            \addplot[black, thick, domain=2e-2:2e-1, samples=100] {0.2*x^1.7};
            \addlegendentry{$h^{1.7}$}
            \addplot[black, thick, dashed, domain=2e-2:2e-1, samples=100] {0.3*x^0.8};
            \addlegendentry{$h$}
        \end{loglogaxis}
    \end{tikzpicture}
	\caption{Convergence test for high-fidelity (finite element based) tangent plane scheme (HF), \textbf{OG-3x} and \textbf{SS-OG-3x} under $h$-refinement of the snapshots.
	Left: Error with respect to the mesh size.
	Right: Minimum inf-sup constant (of the linear systems solved over time steps) with respect to the mesh size.}
	\label{fig:err_infsup_h}
\end{figure}

\subsection{High-dimensional Parameter Space} \label{sec:numerics_nd}
In this section, we consider a parameter space of dimension
higher than one, i.e. $\bGamma = \R^s$ for some $s\in\N$. 
It should be noted that choosing a finite $s$ is equivalent to truncating the Lévy-Ciesielski expansion~\eqref{eq:LCP} to its first $s$ terms. 
Assuming that $s$ is large enough, the resulting \emph{truncation error} is negligible compared to other error contributions.

Other than the dimension of the parameter space, we consider the same setting  (i.e., domain $D$, final time $T$, initial condition $\bm^0$, Gilbert damping parameter $\alpha$, spatial noise distribution $\bg$, external magnetic field $\bH_{\text{ext}}$) and discretization parameters (time step size, mesh size, discretization orders in time and space) as in Section~\ref{sec:setup}.

\subsubsection{Parametric Dimension-Robustness of Galerkin POD-TPS}
We test the robustness of the POD-TPS algorithm with respect to the parameter space dimension. 
To this end, we consider parametric dimensions $s=1, 10, 100$.
For each value of $s$, we sample $N_S = 128$ i.i.d. standard Normal random vectors $\by^{(1)}, \dots, \by^{(N_S)} \in \bGamma$. 
Then, using the TPS as a high-fidelity model, we compute the corresponding snapshots for the magnetization, its velocity, and the Lagrange multipliers.
Again, following Section \ref{eq:empirical_POD}, we compute
reduced bases using POD for three involved quantities.

The singular values, displayed in the left column of Figure~\ref{fig:svs_proj_err_vs_param_dim}, tend to decay more slowly as $s$ increases. However, for $s = 10$ and $s = 100$ these are closer to each other than to the ones for $s = 1$. This behavior suggests convergence of the singular values as the parametric dimension increases. 

We compare the accuracy of the reduced spaces for the 
magnetizations, velocities and Lagrange multipliers and
for different parametric dimensions by computing the
projection error as defined in Section \ref{eq:error_metric}.
As in the previous section, the error computation is based on $M=30$ test parameter samples $\set{\by^{\text{(test)}}_1, \dots, \by^{\text{(test)}}_M} \in \bGamma$ (different from the ones used for the snapshots).

In the right column of Figure~\ref{fig:svs_proj_err_vs_param_dim} we plot, for the previously selected values of $s$, the projection error as a function of the reduced space dimension. As per usual, this is repeated for magnetization, its velocities, and the Lagrange multipliers. 
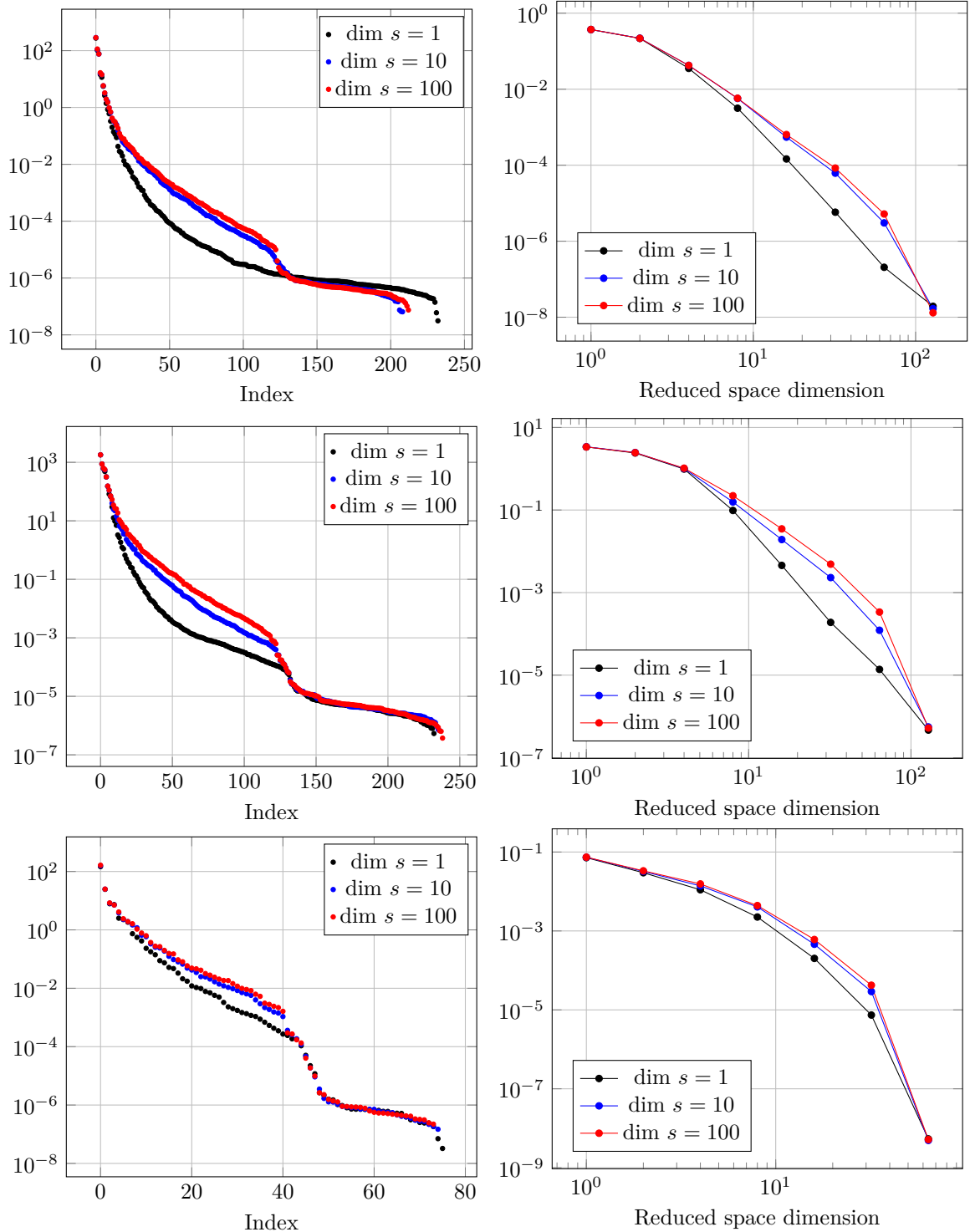
\begin{figure}
	\center
    %
    %
    \begin{tikzpicture}
        \begin{semilogyaxis}[
            grid=major,
            width=0.52\linewidth,
            xlabel={Index},
            ]
            \addplot[only marks, mark=*, mark size=1pt, color=black] table[col sep=comma, x expr=\coordindex, y index={0}]{data_plots/ND/singular_values_m_dim_y_1.csv};
            \addlegendentry{$\text{dim } s = 1$}
            \addplot[only marks, mark=*, mark size=1pt, color=blue] table[col sep=comma, x expr=\coordindex, y index={0}]{data_plots/ND/singular_values_m_dim_y_10.csv};
            \addlegendentry{$\text{dim } s = 10$}
            \addplot[only marks, mark=*, mark size=1pt, color=red] table[col sep=comma, x expr=\coordindex, y index={0}]{data_plots/ND/singular_values_m_dim_y_100.csv};
            \addlegendentry{$\text{dim } s = 100$}
        \end{semilogyaxis}
    \end{tikzpicture}
    \begin{tikzpicture}
        \begin{loglogaxis}[
            grid=major,
            xlabel={Reduced space dimension},
            width=0.52\linewidth,
            legend style={at={(0.05,0.05)}, anchor=south west}
            ]
            \addplot[mark=*, mark size=1.5pt] table[col sep=comma, x index={0}, y index={1}]{data_plots/ND/conv_proj_m_dim_y_1.csv};
            \addlegendentry{$\text{dim } s = 1$}
            \addplot[mark=*, mark size=1.5pt, blue] table[col sep=comma, x index={0}, y index={1}]{data_plots/ND/conv_proj_m_dim_y_10.csv};
            \addlegendentry{$\text{dim } s = 10$}
            \addplot[mark=*, mark size=1.5pt, red] table[col sep=comma, x index={0}, y index={1}]{data_plots/ND/conv_proj_m_dim_y_100.csv};
            \addlegendentry{$\text{dim } s = 100$}
        \end{loglogaxis}
    \end{tikzpicture}
    %
    \begin{tikzpicture}
        \begin{semilogyaxis}[
            grid=major,
            width=0.52\linewidth,
            xlabel={Index},
            ]
            \addplot[only marks, mark=*, mark size=1pt, color=black] table[col sep=comma, x expr=\coordindex, y index={0}]{data_plots/ND/singular_values_v_dim_y_1.csv};
            \addlegendentry{$\text{dim } s = 1$}
            \addplot[only marks, mark=*, mark size=1pt, color=blue] table[col sep=comma, x expr=\coordindex, y index={0}]{data_plots/ND/singular_values_v_dim_y_10.csv};
            \addlegendentry{$\text{dim } s = 10$}
            \addplot[only marks, mark=*, mark size=1pt, color=red] table[col sep=comma, x expr=\coordindex, y index={0}]{data_plots/ND/singular_values_v_dim_y_100.csv};
            \addlegendentry{$\text{dim } s = 100$}
        \end{semilogyaxis}
    \end{tikzpicture}
    \begin{tikzpicture}
        \begin{loglogaxis}[
            grid=major,
            xlabel={Reduced space dimension},
            width=0.52\linewidth,
            legend style={at={(0.05,0.05)}, anchor=south west}
            ]
            \addplot[mark=*, mark size=1.5pt] table[col sep=comma, x index={0}, y index={1}]{data_plots/ND/conv_proj_v_dim_y_1.csv};
            \addlegendentry{$\text{dim } s = 1$}
            \addplot[mark=*, mark size=1.5pt, blue] table[col sep=comma, x index={0}, y index={1}]{data_plots/ND/conv_proj_v_dim_y_10.csv};
            \addlegendentry{$\text{dim } s = 10$}
            \addplot[mark=*, mark size=1.5pt, red] table[col sep=comma, x index={0}, y index={1}]{data_plots/ND/conv_proj_v_dim_y_100.csv};
            \addlegendentry{$\text{dim } s = 100$}
        \end{loglogaxis}
    \end{tikzpicture}
    %
    \begin{tikzpicture}
        \begin{semilogyaxis}[
            grid=major,
            width=0.52\linewidth,
            xlabel={Index},
            ]
            \addplot[only marks, mark=*, mark size=1pt, color=black] table[col sep=comma, x expr=\coordindex, y index={0}]{data_plots/ND/singular_values_l_dim_y_1.csv};
            \addlegendentry{$\text{dim } s = 1$}
            \addplot[only marks, mark=*, mark size=1pt, color=blue] table[col sep=comma, x expr=\coordindex, y index={0}]{data_plots/ND/singular_values_l_dim_y_10.csv};
            \addlegendentry{$\text{dim } s = 10$}
            \addplot[only marks, mark=*, mark size=1pt, color=red] table[col sep=comma, x expr=\coordindex, y index={0}]{data_plots/ND/singular_values_l_dim_y_100.csv};
            \addlegendentry{$\text{dim } s = 100$}
        \end{semilogyaxis}
    \end{tikzpicture}
    \begin{tikzpicture}
        \begin{loglogaxis}[
            grid=major,
            xlabel={Reduced space dimension},
            width=0.52\linewidth,
            legend style={at={(0.05,0.05)}, anchor=south west}
            ]
            \addplot[mark=*, mark size=1.5pt] table[col sep=comma, x index={0}, y index={1}]{data_plots/ND/conv_proj_l_dim_y_1.csv};
            \addlegendentry{$\text{dim } s = 1$}
            \addplot[mark=*, mark size=1.5pt, blue] table[col sep=comma, x index={0}, y index={1}]{data_plots/ND/conv_proj_l_dim_y_10.csv};
            \addlegendentry{$\text{dim } s = 10$}
            \addplot[mark=*, mark size=1.5pt, red] table[col sep=comma, x index={0}, y index={1}]{data_plots/ND/conv_proj_l_dim_y_100.csv};
            \addlegendentry{$\text{dim } s = 100$}
        \end{loglogaxis}
    \end{tikzpicture}
	\caption{Robustness of the POD for parametric dimensions $s = 1,10,100$. Left column: Singular values resulting from the POD of magnetizations, their velocities, and Lagrange multipliers respectively. Right column: Projection errors as a function of the number of reduced basis function for magnetizations, their velocities, and Lagrange multipliers respectively.}
	\label{fig:svs_proj_err_vs_param_dim}
\end{figure}

\subsubsection{Convergence of Sparse Grid-Reduced Basis Projections}

We test the SG-RBP method introduced in 
Section \ref{sec:sg} on the parametric LLG equation with final time $T=0.2$. The rest of the problem's data remains as in Section~\ref{sec:setup}. We compute the high-fidelity solution with time step $\tau=10^{-3}$ and mesh size $h=0.125$.

In Figure~\ref{fig:results_sg_proj_HF}, we study the approximation error (with respect to $M=30$ high-fidelity samples) and the number of active sparse grid dimensions as the number of sparse grid nodes increases. We also vary the number of reduced basis functions by truncating the reduced basis generated offline with different tolerances following Remark \ref{eq:criterium}.
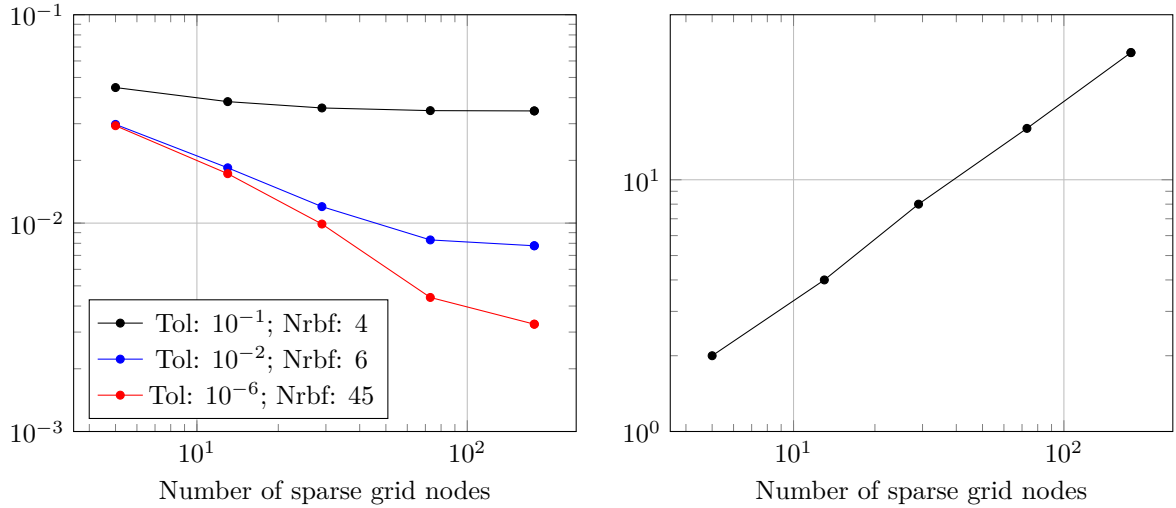
\begin{figure}
    \begin{tikzpicture}
        \begin{loglogaxis}[
            grid=major,
            xlabel={Number of sparse grid nodes},
            legend pos=south west,
            ymin=1.e-3, ymax=1.e-1,
            width=0.53\linewidth
        ]
            \addplot[mark=*, mark size=1.5pt, black] table[col sep=comma, x index={0}, y index={3}
            ]{data_plots/ND/convergence_sg_tol_1e-1.csv};
            \addlegendentry{Tol: $10^{-1}$; Nrbf: 4}
            \addplot[mark=*, mark size=1.5pt, blue] table[col sep=comma, x index={0}, y index={3}
            ]{data_plots/ND/convergence_sg_tol_1e-2.csv};
            \addlegendentry{Tol: $10^{-2}$; Nrbf: 6}
            \addplot[mark=*, mark size=1.5pt, red] table[col sep=comma, x index={0}, y index={3}
            ]{data_plots/ND/convergence_sg_tol_1e-6.csv};
            \addlegendentry{Tol: $10^{-6}$; Nrbf: 45 }
        \end{loglogaxis}
    \end{tikzpicture}
    \hspace{0.02\textwidth}
    \begin{tikzpicture}
        \begin{loglogaxis}[
            grid=major,
            xlabel={Number of sparse grid nodes},
            ymin=1.e0,
            ytick={1.e0, 1.e1},
            width=0.53\linewidth
        ]
            \addplot[mark=*, mark size=1.5pt, black] table[col sep=comma, x index={0}, y index={1}
            ]{data_plots/ND/convergence_sg_tol_1e-1.csv};
        \end{loglogaxis}
    \end{tikzpicture}
	\caption{Convergence of SG-RBP. 
    Left: Error with respect to the number of sparse grid nodes, for different magnetization reduced basis tolerance (Tol) and corresponding number of reduced basis functions (Nrbf).
    Right: Number of effective dimensions of the sparse grids as function of the number of sparse grid nodes.}
	\label{fig:results_sg_proj_HF}
\end{figure}

As expected, the error has an \emph{additive} structure with respect to sparse grid and reduced basis parameters. 
This is apparent because the sparse grid error dominates until the reduced basis accuracy is reached, at which point a stagnation is visible.

\subsection{A Switching Example with Non-constant Initial Condition}\label{sec:numerics_switching}
In this section, we carry out a numerical study of the methods introduced above in a ``random switching'' setup, which is both challenging and relevant in applications. See e.g.~\cite{banas2013computational, scholz2001Micromagnetic} for a more detailed treatment of the topic.

We consider a physical domain $D$ as in \eqref{eq:domain_D},
final time $T=1$, Gilbert damping parameter $\alpha=1.4$,
initial condition
\begin{equation}
	\bm^0(\bx) 
	= 
	\left(
		C (x_1 - 0.5),\;
		C (x_2 - 0.5),\;
		\sqrt{1 - C^2 (x_1 - 0.5)^2 - C^2 (x_2 - 0.5)^2}
	\right)^\top,
\end{equation}
where $C = 0.9$, spacial noise distribution
\begin{equation}
	\bg(\bx) 
	= 
	\left(
		0.45 \sin(\pi x_1),\;
		0.45 \sin(\pi x_2),\;
		\sqrt{1 - 0.45^2 \sin^2(\pi x_1) - 0.45^2 \sin^2(\pi x_2)}
	\right)^\top,
\end{equation}
and external magnetic field $\bH(\bx) = (0, 0, -1)$.

In Figure~\ref{fig:ex_sample_paths_switch}, we show two sample paths of the Brownian motion and the corresponding magnetization at three fixed time stamps.
\begin{figure}
\centering
\setlength{\tabcolsep}{0pt} 
\renewcommand{\arraystretch}{0} 
\begin{tabular}{cccc}
	\begin{tikzpicture}
		\begin{axis}[
			width=0.259\linewidth,
			xlabel={$t$},
			xmin=0, xmax=1,
			ymin=-1.5, ymax=2.1,
			grid=major,
			thick,
			axis line style={-},
			axis x line=bottom,
			axis y line=left,
			xtick={0,0.5,1},
		]
			\addplot+[blue, mark=none] table [x index={0}, y index={1}, col sep=comma] {data_plots/switch_ncIC/metrics_tt_seed_2.csv};
		\end{axis}
	\end{tikzpicture}
	&
	\includegraphics[width=0.259\linewidth,trim=20 0 20 0,clip]{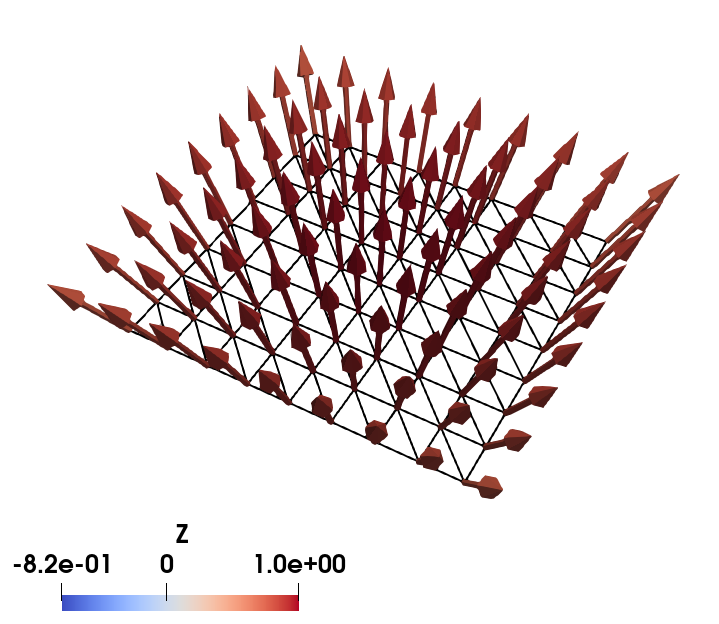}
	&
	\includegraphics[width=0.259\linewidth,trim=20 0 20 0,clip]{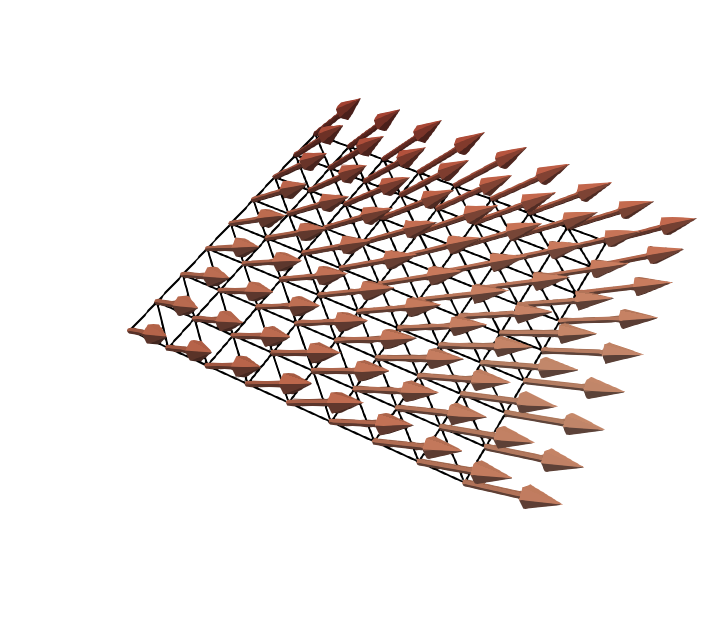}
	&
	\includegraphics[width=0.259\linewidth,trim=20 0 20 0,clip]{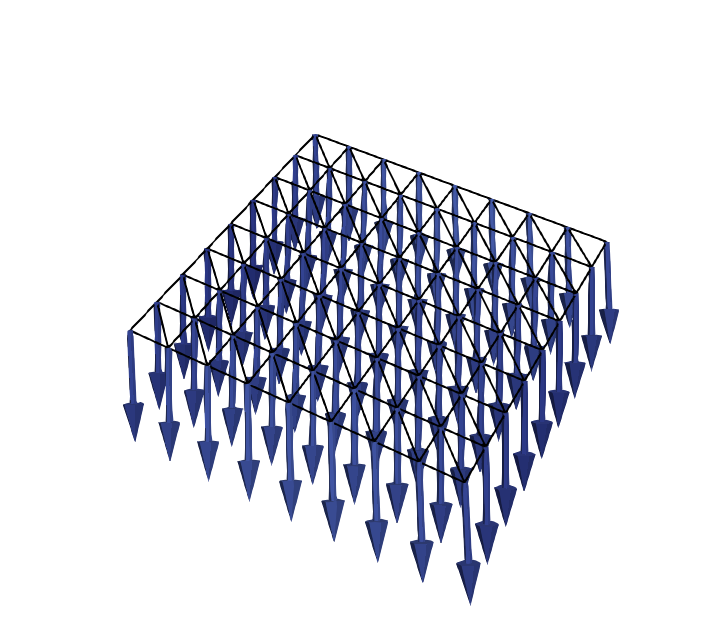}
	\\
	\begin{tikzpicture}
		\begin{axis}[
			width=0.259\linewidth,
			height=0.25\textwidth,
			xlabel={$t$},
			xmin=0, xmax=1,
			ymin=-1.5, ymax=2.1,
			grid=major,
			thick,
			axis line style={-},
			axis x line=bottom,
			axis y line=left,
			xtick={0,0.5,1},
		]
			\addplot+[blue, mark=none] table [x index={0}, y index={1}, col sep=comma] {data_plots/switch_ncIC/metrics_tt_seed_3.csv};
		\end{axis}
	\end{tikzpicture}
	&
	\includegraphics[width=0.259\linewidth, trim=20 0 20 0,clip]{data_plots/switch_ncIC/m_S2_t0.png}
	&
	\includegraphics[width=0.259\linewidth, trim=20 0 20 0,clip]{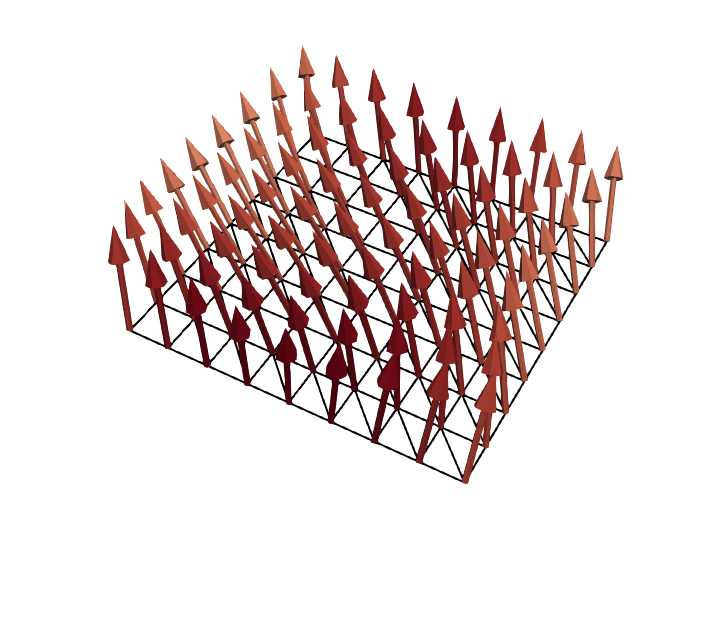}
	&
	\includegraphics[width=0.259\linewidth, trim=20 0 20 0,clip]{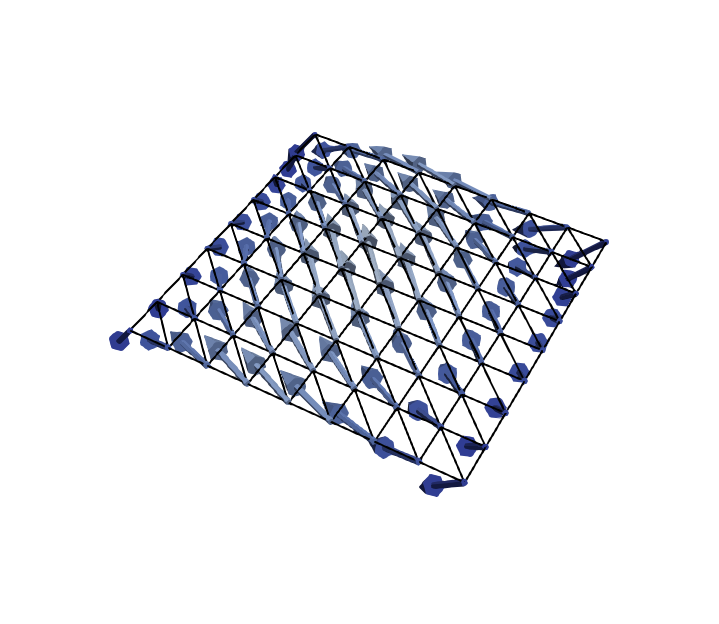}
\end{tabular}
\caption{In each row, we plot 1 sample path of the Brownian motion and the corresponding random magnetization in the switching setup at times $t = 0, 0.5, 1$. The magnetization is color-coded with the magnitude of the z component (blue corresponds to -1 and red to 1).}
\label{fig:ex_sample_paths_switch}
\end{figure}

We observe that the dynamic changes dramatically for different sample paths. In particular, the switching begins at different times and along different directions. 
Moreover, a slightly different perturbation close to the initial time can lead to dramatically different switching directions. This intuitively means that the SLLG switching dynamic does not depend continuously on the random perturbations, thus making the parametric dependence of the solution difficult to approximate. 

In Figure~\ref{fig:switch_cnIC_metrics}, we additionally show some basic metrics related to the time evolution of the magnetization field.

We use the setup described above as a benchmark to test the schemes we have previously introduced, i.e. Galerkin POD-TPS and SG-RBP.
\begin{figure}
	\definecolor{mycolor1}{rgb}{0.121,0.466,0.705}
	\definecolor{mycolor2}{rgb}{1.000,0.498,0.054}
	\definecolor{mycolor3}{rgb}{0.172,0.627,0.172}
	\definecolor{mycolor4}{rgb}{0.839,0.153,0.157}
	\definecolor{mycolor5}{rgb}{0.580,0.404,0.741}
	\definecolor{mycolor6}{rgb}{0.549,0.337,0.294}
	\definecolor{mycolor7}{rgb}{0.890,0.466,0.760}
	\definecolor{mycolor8}{rgb}{0.498,0.498,0.498}
	\definecolor{mycolor9}{rgb}{0.737,0.741,0.133}
	\definecolor{mycolor10}{rgb}{0.090,0.745,0.811}

    \begin{tikzpicture}
        \begin{axis}[
            xlabel={$t$},
            xmin=0, xmax=1, ymin=-1, ymax=1,
            grid=major,
            thick,
            axis line style={-},
            axis x line=bottom, axis y line=left,
            width=0.35\linewidth,
            tick label style={font=\scriptsize},
            label style={font=\scriptsize},
        ]
            \addplot[smooth, thick, no markers,color=mycolor1] table [x index={0}, y index={1}, col sep=comma] {data_plots/switch_ncIC/aver_mz_tt_HFsamples.csv};
            \addplot[smooth, thick, no markers,color=mycolor2] table [x index={0}, y index={2}, col sep=comma] {data_plots/switch_ncIC/aver_mz_tt_HFsamples.csv};
            \addplot[smooth, thick, no markers,color=mycolor3] table [x index={0}, y index={3}, col sep=comma] {data_plots/switch_ncIC/aver_mz_tt_HFsamples.csv};
            \addplot[smooth, thick, no markers,color=mycolor4] table [x index={0}, y index={4}, col sep=comma] {data_plots/switch_ncIC/aver_mz_tt_HFsamples.csv};
            \addplot[smooth, thick, no markers,color=mycolor5] table [x index={0}, y index={5}, col sep=comma] {data_plots/switch_ncIC/aver_mz_tt_HFsamples.csv};
            \addplot[smooth, thick, no markers,color=mycolor6] table [x index={0}, y index={6}, col sep=comma] {data_plots/switch_ncIC/aver_mz_tt_HFsamples.csv};
            \addplot[smooth, thick, no markers,color=mycolor7] table [x index={0}, y index={7}, col sep=comma] {data_plots/switch_ncIC/aver_mz_tt_HFsamples.csv};
            \addplot[smooth, thick, no markers,color=mycolor8] table [x index={0}, y index={8}, col sep=comma] {data_plots/switch_ncIC/aver_mz_tt_HFsamples.csv};
            \addplot[smooth, thick, no markers,color=mycolor9] table [x index={0}, y index={9}, col sep=comma] {data_plots/switch_ncIC/aver_mz_tt_HFsamples.csv};
            \addplot[smooth, thick, no markers,color=mycolor10] table [x index={0}, y index={10}, col sep=comma] {data_plots/switch_ncIC/aver_mz_tt_HFsamples.csv};
        \end{axis}
    \end{tikzpicture}
    \hfill
    \begin{tikzpicture}
        \begin{axis}[ 
            xlabel={$t$},
            xmin=0, xmax=1, ymin=0,
            grid=major,
            thick,
            axis line style={-},
            axis x line=bottom, axis y line=left,
            width=0.35\linewidth,
            tick label style={font=\scriptsize},
            label style={font=\scriptsize},
        ]
            \addplot[smooth, thick, no markers,color=mycolor1] table [x index={0}, y index={1}, col sep=comma] {data_plots/switch_ncIC/dirichlet_energy_tt_HFsamples.csv};
            \addplot[smooth, thick, no markers,color=mycolor2] table [x index={0}, y index={2}, col sep=comma] {data_plots/switch_ncIC/dirichlet_energy_tt_HFsamples.csv};
            \addplot[smooth, thick, no markers,color=mycolor3] table [x index={0}, y index={3}, col sep=comma] {data_plots/switch_ncIC/dirichlet_energy_tt_HFsamples.csv};
            \addplot[smooth, thick, no markers,color=mycolor4] table [x index={0}, y index={4}, col sep=comma] {data_plots/switch_ncIC/dirichlet_energy_tt_HFsamples.csv};
            \addplot[smooth, thick, no markers,color=mycolor5] table [x index={0}, y index={5}, col sep=comma] {data_plots/switch_ncIC/dirichlet_energy_tt_HFsamples.csv};
            \addplot[smooth, thick, no markers,color=mycolor6] table [x index={0}, y index={6}, col sep=comma] {data_plots/switch_ncIC/dirichlet_energy_tt_HFsamples.csv};
            \addplot[smooth, thick, no markers,color=mycolor7] table [x index={0}, y index={7}, col sep=comma] {data_plots/switch_ncIC/dirichlet_energy_tt_HFsamples.csv};
            \addplot[smooth, thick, no markers,color=mycolor8] table [x index={0}, y index={8}, col sep=comma] {data_plots/switch_ncIC/dirichlet_energy_tt_HFsamples.csv};
            \addplot[smooth, thick, no markers,color=mycolor9] table [x index={0}, y index={9}, col sep=comma] {data_plots/switch_ncIC/dirichlet_energy_tt_HFsamples.csv};
            \addplot[smooth, thick, no markers,color=mycolor10] table [x index={0}, y index={10}, col sep=comma] {data_plots/switch_ncIC/dirichlet_energy_tt_HFsamples.csv};
        \end{axis}
    \end{tikzpicture}
    \hfill
    \includegraphics[width=0.35\linewidth]{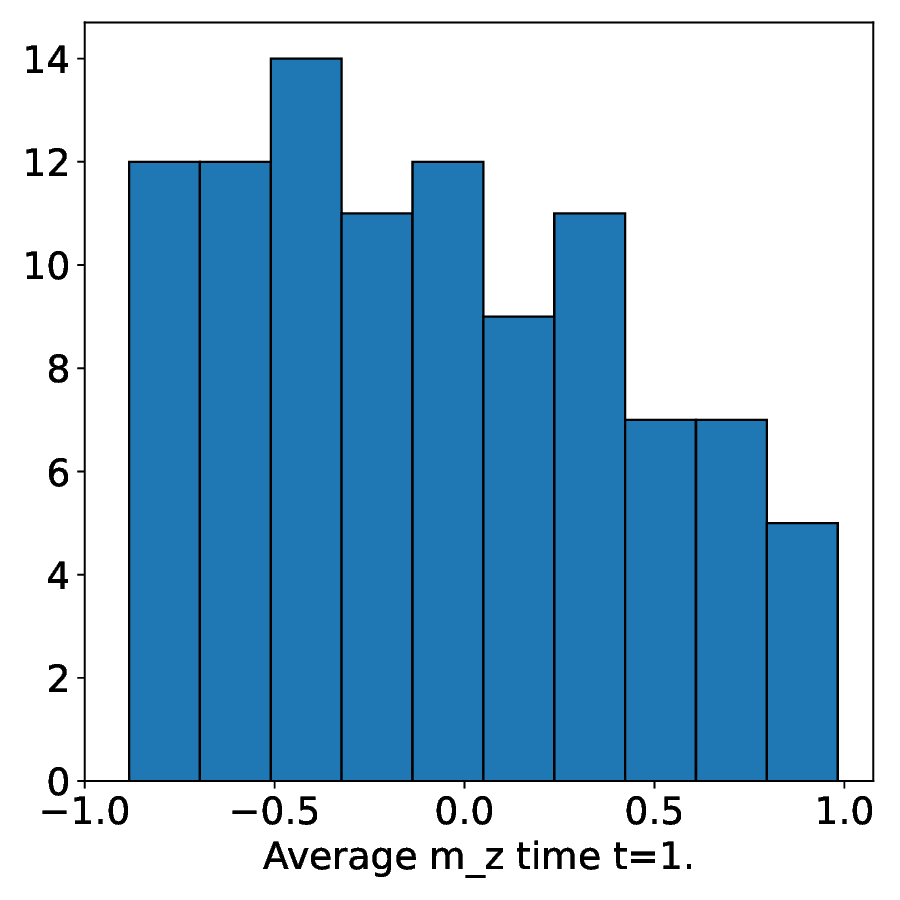}
	\caption{Metrics from repeated sampling of the SLLG dynamics (computed with high fidelity TPS).
		Left: Space-averaged $z$ component of the magnetization $\int_D m_z(\by, t, \bx) \text{d}x$ with respect to time $t\in I$;
		Center: Dirichlet energy $\int_D \seminorm{\nabla \bm(\by, t, \bx)}{}^2 \text{d}x$ with respect to time $t\in I$;
		Right: Histogram of the space-averaged $z$ component of the magnetization at final time $T$, $\int_D m_z(\by, T, \bx)\text{d}\bx$.
	} 
	\label{fig:switch_cnIC_metrics}
\end{figure}
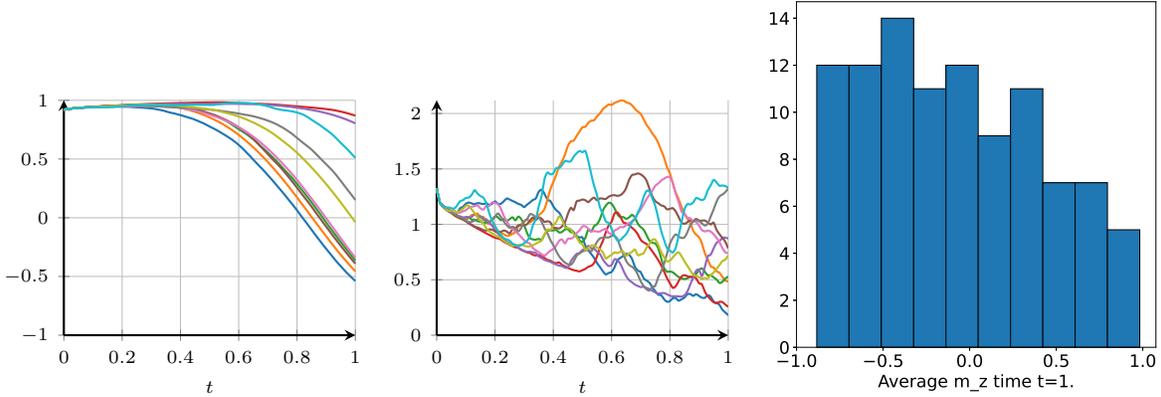

The tolerance on the relative accuracy of the reduced basis is $\epsilon_{\normalfont\text{POD}}^2 = 10^{-6}$.
Both methods use the same reduced basis, computed with 
$N_S=64$ snapshots of the switching SLLG dynamics corresponding to the parameter vectors $\by\in \R^{100}$ with i.i.d. standard Normal components.
The snapshots are obtained with the high-fidelity TPS with 
order 1 BDF, time step $\tau=10^{-3}$, 
order 1 finite elements over a structured triangular mesh of $D$ with 128 elements and mesh size $h=0.125$.

The first five reduced basis functions for the magnetization and the singular values  for the magnetization, its velocity, and the Lagrange multipliers, are visualized in Figure~\ref{fig:5_rbfs_switch}. The reduced basis functions reflect the magnetization dynamics: The first is close to the initial condition. The following 4 have often small $z$ component and are oriented in different directions within the $xy$ plane. Overall, their linear combinations already give a good approximation of magnetizations with different orientations.

\begin{figure}
	\begin{minipage}[t]{0.34\linewidth}
		\vspace{0pt}
		\begin{tikzpicture}
			\begin{semilogyaxis}[
				xmin=0, xmax=100, 
				grid=major, 
				width=\linewidth, 
				height=1.5\linewidth,
				legend style={font=\small}
			]
				\addplot[only marks, mark=*, mark size=1.5pt] table[col sep=comma, x expr=\coordindex, y index={0}]{data_plots/switch_ncIC/singular_values_m.csv};
				\addlegendentry{Magnetization}
				\addplot[only marks, mark=*, mark size=1.5pt, blue] table[col sep=comma, x expr=\coordindex, y index={0}]{data_plots/switch_ncIC/singular_values_v.csv};
				\addlegendentry{Velocity}
				\addplot[only marks, mark=*, mark size=1.5pt, red] table[col sep=comma, x expr=\coordindex, y index={0}]{data_plots/switch_ncIC/singular_values_l.csv};
				\addlegendentry{Lagrange Mult's}
			\end{semilogyaxis}
		\end{tikzpicture}
	\end{minipage}%
	\hfill
	\begin{minipage}[t]{0.64\linewidth}
		\vspace{0pt}
		\includegraphics[width=0.49\textwidth, trim={0cm 0.5cm 0.8cm 2.5cm},clip]{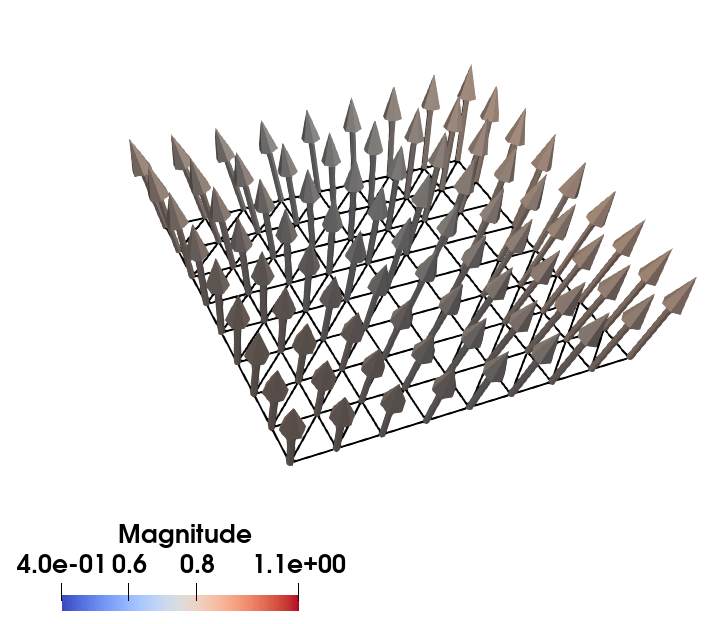}
		\includegraphics[width=0.49\textwidth, trim={2cm 0.5cm 0.8cm 2.5cm},clip]{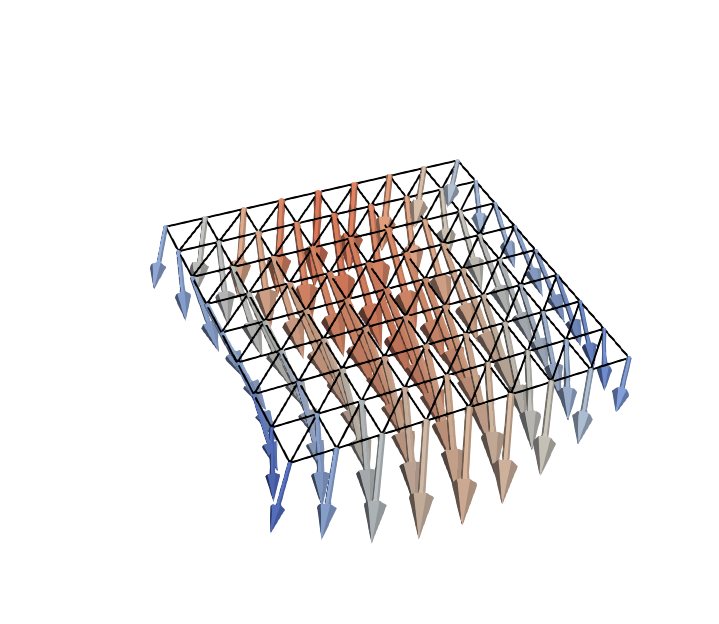}
		\includegraphics[width=0.49\textwidth, trim={0cm 0.5cm 0.8cm 2.5cm},clip]{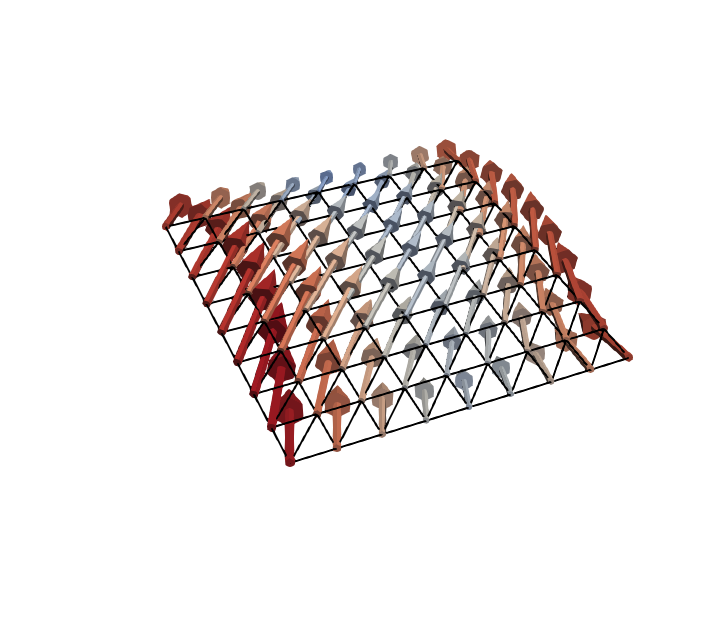}
		\includegraphics[width=0.49\textwidth, trim={2cm 0.5cm 0.8cm 2.5cm},clip]{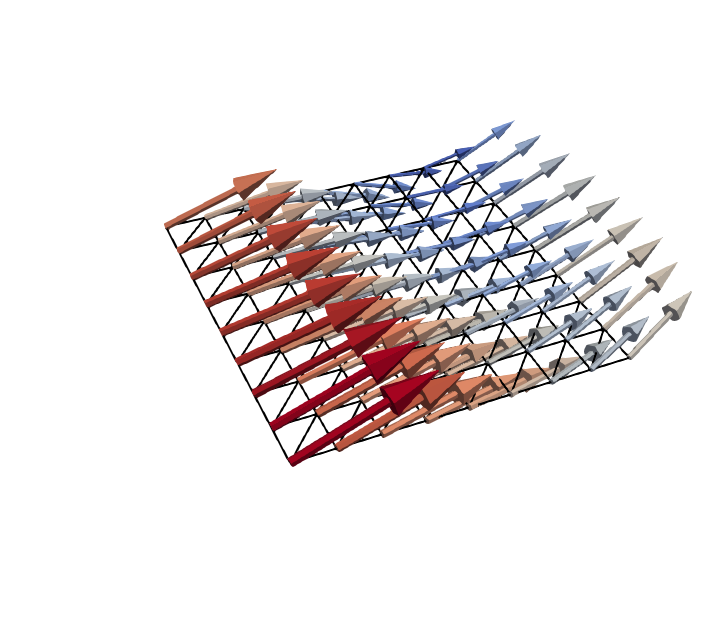}
	\end{minipage}
	\caption{Left: first 100 singular values for each variable from the switching example with non-constant initial condition.  Right: 4 magnetization reduced basis functions corresponding to the 4 leading singular values. The vector fields are color-coded with the magnitude.}
	\label{fig:5_rbfs_switch}
\end{figure}
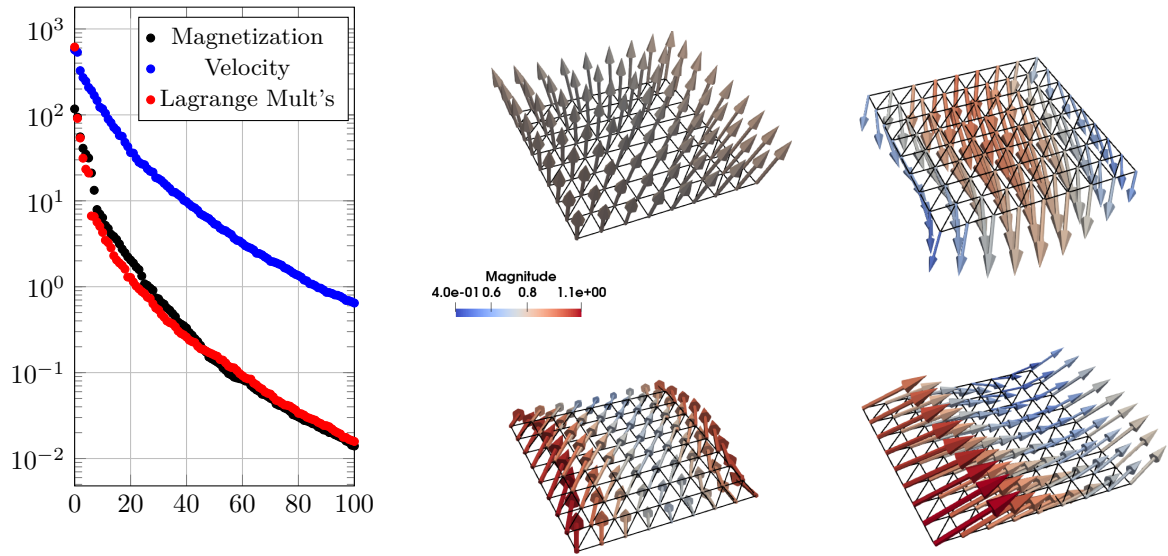

We consider the supremizer-stabilized POD-TPS method with the same number of velocities and Lagrange multipliers reduced basis functions.
We compare it with SG-RBP and select the scheme's parameters to obtain about $N_S=64$ sparse grid nodes, i.e. as many as the snapshots used to compute the reduced bases above. Finally, we recall that the high fidelity scheme is always inf-sup stable.

In Figure~\ref{fig:compare_approx_fix_t}, we visually compare the approximation power of the schemes for a single online sample. While the POD-TPS approximation is visually indistinguishable from the high-fidelity solution, the SG-RBP approximation is largely inaccurate. Intuitively, we interpret this as follows: POD-TPS approximates better the energy of the system because it uses the same time-marching as the
high-fidelity model. Conversely, the SG-RBP is a purely ``data-driven'' approximation that does not make assumptions regarding the time evolution, thus is highly dissipative (on this challenging example) and can only converge with very many samples.
\begin{figure}
	\includegraphics[width=0.33\linewidth,trim=20 20 20 20,clip]{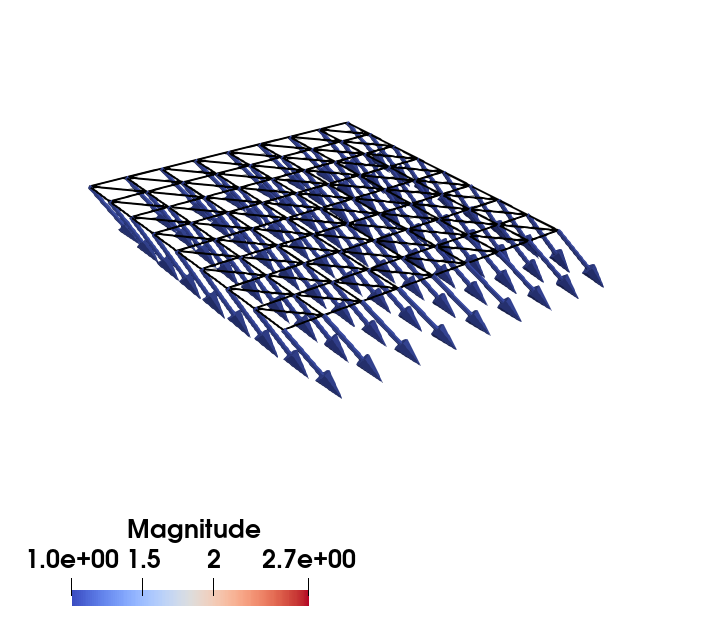}
	\includegraphics[width=0.33\linewidth,trim=20 20 20 20,clip]{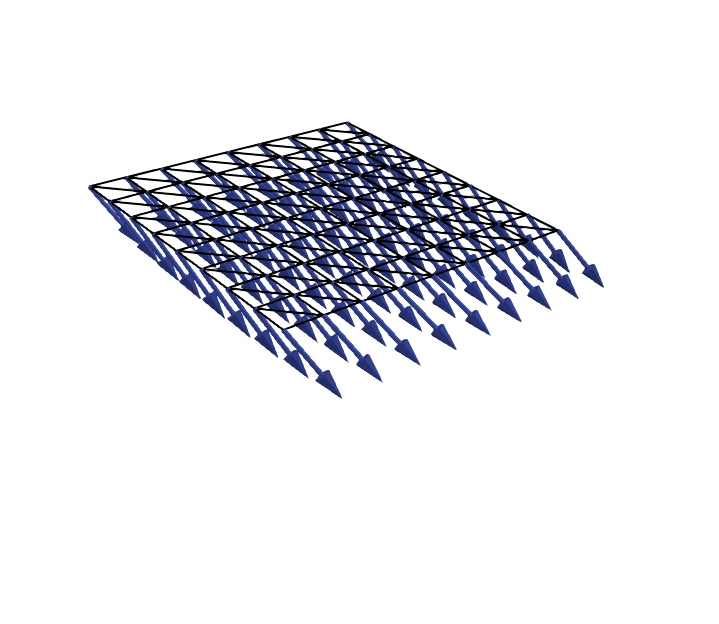}
	\includegraphics[width=0.33\linewidth,trim=20 20 20 20,clip]{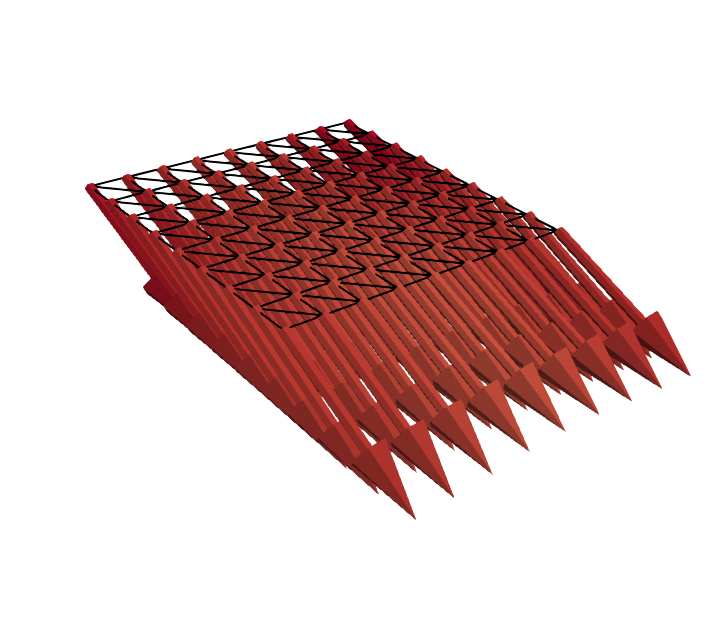}
	\caption{Visual comparison of POD-TPS and SG-RBP approximations on the switching example. Each image represents the magnetization at time $t=1$. Left: High-fidelity solution; Center: POD-TPS approximation; Right: SG-RBP approximation.}
	\label{fig:compare_approx_fix_t}
\end{figure}

In Figure~\ref{fig:compare_pod-tps_sg_rbp_tt}, we compare the two methods through metrics related to time. In particular, denoting $\bm$ the high-fidelity solution and $\bm_h$ the approximation with one of the two methods, we consider:
\begin{itemize}
	\item Unit modulus error: i.e. 
        \begin{equation}
            \int_D \seminorm{1-\seminorm{\bm_h(\by, t, \bx)}{}^2}{} \text{d}\bx,
        \end{equation}
	\item $H^1(D)^3$-error, i.e. 
		\begin{equation}
		    \sqrt{\int_D \seminorm{ (\bm-\bm_h)(\by, t, \bx)}{}^2 + \seminorm{\nabla (\bm-\bm_h)(\by, t, \bx)}{}^2 \text{d}\bx},
		\end{equation}
	\item Dirichlet energy, i.e. 
		\begin{equation}
		    \sqrt{\int_D \seminorm{\nabla \bm(\by, t, \bx)}{}^2 \text{d}\bx},
		\end{equation}
	\item Space-averaged $z$-component of the magnetization, i.e.
        \begin{equation}
            \int_D m_z(\by, t, \bx) \text{d} \bx.
        \end{equation}
\end{itemize}
	\begin{figure}
		\centering
        \begin{tikzpicture}
            \begin{semilogyaxis}[ 
                xlabel={$t$},
                grid=major,
                thick,
                axis line style={-},
                axis x line=bottom, axis y line=left,
                width=0.48\linewidth,
                legend style={at={(0.02,0.98)},anchor=north west}
                ]
                \addplot[smooth, ultra thick, no markers,color=blue] table [x index={0}, y index={3}, col sep=comma] {data_plots/switch_ncIC/time_metrics_POD.csv};
                \addlegendentry{POD-TPS}
                \addplot[smooth, ultra thick,color=red] table [x index={0}, y index={2}, col sep=comma] {data_plots/switch_ncIC/time_metrics_SG.csv};
                \addlegendentry{SG-RBP}
            \end{semilogyaxis}
        \end{tikzpicture}
        \begin{tikzpicture}
            \begin{semilogyaxis}[ 
                xlabel={$t$},
                grid=major,
                thick,
                axis line style={-},
                axis x line=bottom, axis y line=left,
                width=0.48\linewidth,
                legend style={at={(0.98,0.02)},anchor=south east}
                ]
                \addplot[smooth, ultra thick, no markers,color=blue] table [x index={0}, y index={1}, col sep=comma] {data_plots/switch_ncIC/time_metrics_POD.csv};
                \addlegendentry{POD-TPS}
                \addplot[smooth, ultra thick,color=red] table [x index={0}, y index={1}, col sep=comma] {data_plots/switch_ncIC/time_metrics_SG.csv};
                \addlegendentry{SG-RBP}
            \end{semilogyaxis}
        \end{tikzpicture}
        %
        %
        \begin{tikzpicture}
            \begin{axis}[ 
                xlabel={$t$},
                grid=major,
                thick,
                axis line style={-},
                axis x line=bottom, axis y line=left,
                width=0.48\linewidth,
                ]
                \addplot[smooth, ultra thick, color=blue] table [x index={0}, y index={4}, col sep=comma] {data_plots/switch_ncIC/time_metrics_POD.csv};
                \addlegendentry{POD-TPS}
                \addplot[smooth, ultra thick, dashed, color=orange] table [x index={0}, y index={5}, col sep=comma] {data_plots/switch_ncIC/time_metrics_POD.csv};
                \addlegendentry{HF}
                \addplot[smooth, ultra thick,color=red] table [x index={0}, y index={3}, col sep=comma] {data_plots/switch_ncIC/time_metrics_SG.csv};
                \addlegendentry{SG-RBP}
            \end{axis}
        \end{tikzpicture}
        \begin{tikzpicture}
            \begin{axis}[ 
                xlabel={$t$},
                grid=major,
                 thick,
                axis line style={-},
                axis x line=bottom, axis y line=left,
                width=0.48\linewidth,
                legend style={at={(0.02,0.02)},anchor=south west}
                ]
                \addplot[smooth, ultra thick, no markers,color=blue] table [x index={0}, y index={6}, col sep=comma] {data_plots/switch_ncIC/time_metrics_POD.csv};
                \addlegendentry{POD-TPS}
                \addplot[smooth, ultra thick, dashed, color=orange] table [x index={0}, y index={7}, col sep=comma] {data_plots/switch_ncIC/time_metrics_POD.csv};
                \addlegendentry{HF}
                \addplot[smooth, ultra thick,color=red] table [x index={0}, y index={5}, col sep=comma] {data_plots/switch_ncIC/time_metrics_SG.csv};
                \addlegendentry{SG-RBP}
            \end{axis}
        \end{tikzpicture}
		\caption{Comparison of stabilized POD-TPS and SG-RBP through time-related metrics computed on 1 online sample. All metrics are computed with respect to time $t\in I$ and for a single parameter tuple $\by\in\bGamma$.
		Top Left: Unit modulus error; Top Right: $H^1(D)^3$-error; Bottom Left: Dirichlet energy; Bottom Right: Space-averaged $z$-component of the magnetization.}
		\label{fig:compare_pod-tps_sg_rbp_tt}
	\end{figure}
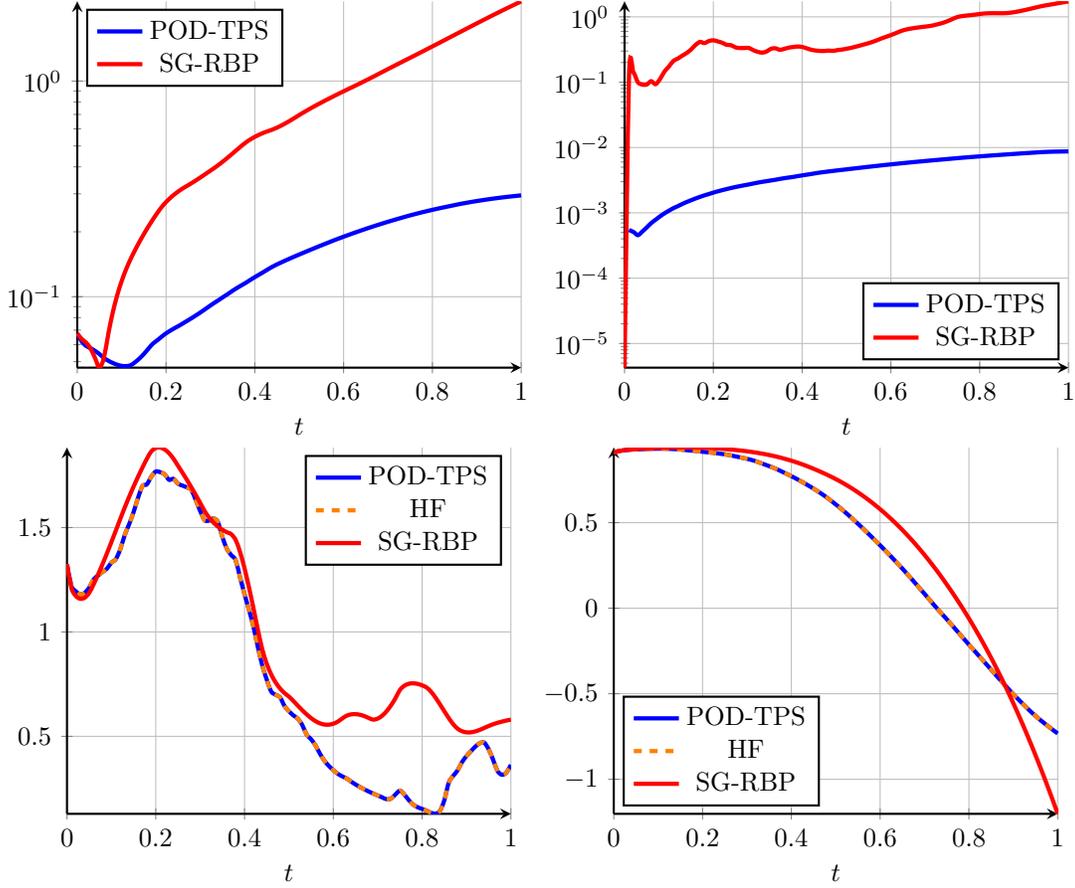

The POD-TPS consistently gives a more accurate approximation. This may be attributed to the more precise energy approximation, clearly visible in the bottom left plot in the figure.

Finally, we compare the two schemes through their error on the $L^2_{\mu}(\bGamma, L^2(I,H^1(D)^3))$ norm, which we approximate with a Monte Carlo mean with $N=30$ normally distributed random samples of the high-fidelity model with time step size $\tau = 10^{-3}$ and mesh size $h=0.125$ (the finite element space $\bV_h^3$ used for the magnetizations and velocities has dimension 867).

In Figure~\ref{fig:compare_schemes_total_error}, we compare the error of the two schemes for the following parameters: 
For the offline snapshots, we use a time step size $\tau = 5\cdot 10^{-3}$ and mesh size $h = 0.125$. 
For the (stabilized) POD-TPS scheme, we decrease the reduced basis truncation tolerance, thus increasing the number of reduced basis functions. We repeat the convergence test with two fixed time steps $\tau = 10^{-3}$ and $\tau = 5\cdot10^{-3}$ to study the effect of time discretization on the online error.
For SG-RBP, we increase the number of collocation points and use a fixed reduced basis obtained with tolerance ${\epsilon^2_{\normalfont\text{POD}}} = 10^{-6}$ and resulting in 70 reduced basis functions. This is enough for its error to be negligible.

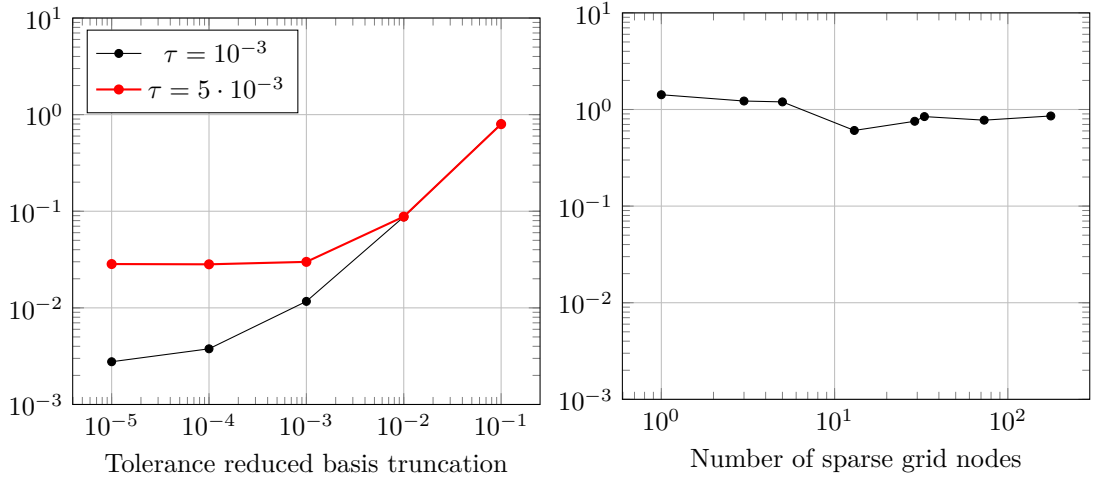
\begin{figure}
	\center
	\begin{tikzpicture}
		\begin{loglogaxis}[
			grid=major,
			xlabel={Tolerance reduced basis truncation},
			legend pos=north west,
			width=0.5\linewidth,
			ymin=1.e-3, ymax=1.e1
		]
			\addplot[mark=*, mark size=1.5pt, color=black] table[col sep=comma, x index={0}, y index={4}
			]{data_plots/switch_ncIC/conv_test_PODTPS_DT_1e-3.csv};
			\addlegendentry{$\tau=10^{-3}$}
			\addplot[mark=*, mark size=1.5pt, thick, color=red] table[col sep=comma, x index={0}, y index={4}
			]{data_plots/switch_ncIC/conv_test_PODTPS_DT_5e-3.csv};
			\addlegendentry{$\tau=5\cdot 10^{-3}$}
		\end{loglogaxis}
	\end{tikzpicture}
	\begin{tikzpicture}
		\begin{loglogaxis}[
			grid=major,
			xlabel={Number of sparse grid nodes},
			legend pos=south west,
			width=0.5\linewidth,
			ymin=1.e-3, ymax=1.e1
		]
			\addplot[mark=*, mark size=1.5pt] table[col sep=comma, x index={0}, y index={3}
			]{data_plots/switch_ncIC/conv_full_norm_SG.csv};
		\end{loglogaxis}
	\end{tikzpicture}
	\caption{Convergence of the approximation error of POD-TPS and SG-RBP schemes used in the online phase.
	Left: POD-TPS error with respect to the tolerance used to truncate the reduced bases computed offline. The online computation is carried out for tho distinct time step sizes.
	Right: SG-RBP error with respect to the number of sparse grid points.}
	\label{fig:compare_schemes_total_error}
\end{figure}

For POD-TPS, we observe a reduction in convergence speed as the number of reduced basis functions increases. This is mostly caused by the time step approximation, as the convergence test with smaller time step size leads, for the same number of reduced basis functions, to smaller errors. Reducing the time step size further would eventually lead to the expected algebraic convergence throughout.

The lack of convergence of SG-RBP, not observed in previous examples, can be attributed to the lack of sparsity of the parametric switching dynamics: The parameter-to-solution map lacks a sufficiently large domain of holomorphic extension. Thus, the error is bound to remain constant unless very many samples are used.
Indeed, all previous examples behaved like a relaxation with random perturbations. By contrast, in the present example, the random noise can cause dramatic changes in the problem's dynamics.

\section{Concluding Remarks}
\label{eq:concluding_remarks}

In this work, we considered the SLLG equation, a time-dependent, nonlinear stochastic PDE used to model the evolution
in time of the magnetization field in a ferromagnetic material and subject to random heat variations. 

Traditionally, the TPS has been used as a reliable, robust scheme to solve the LLG equation, while additionally preserving the unit modulus of the magnetization.
However, due to its high computational cost, caused by finite element discretizations of large dimensions and small time step sizes, it cannot be used for the repeated computation of large ensembles of the dynamics.

After reformulating the SLLG equation into a parametric PDE, we introduced two reduced order/surrogate modeling numerical schemes to alleviate this computational burden. These are:
\begin{itemize}
    \item \textbf{Galerkin POD-Tangent Plane Scheme (POD-TPS)}: 
    Firstly, we sample the solution of the parametric LLG equation in a judiciously selected collection of points in the parameter space. Then, by means of a POD as described in Section \ref{eq:empirical_POD}, we extract bases of small dimension for the velocity field and the Lagrange multiplier of the saddle point formulation.  We propose a novel stabilization method which is proved to generate an inf-sup stable discretization of the saddle point problem arising in the online phase of the Galerkin POD-TPS method. Numerical experiments indicate that this scheme performs exceptionally well compared to the Galerkin POD-TPS with no stabilization.
    \item \textbf{Sparse Grid-Reduced Basis Projection (SG-RBP)}: 
    First, we use the high-fidelity TPS to sample the magnetizations corresponding to the sparse grid collocation nodes in the parameter space. 
    We then compress the magnetizations by projecting onto a reduced basis (computed analogously to the offline phase of the Galerkin POD-TPS scheme). 
    Finally, in the online phase, we interpolate the compressed coordinates vector with a negligible computational burden. While this method does not require any further stabilization, we observe larger errors {than in Galerkin POD-TPS}. This is attributed to the \emph{data driven} nature of the interpolation, which does not guarantee energy preservation.
\end{itemize}


Several challenges remain, particularly in extending these approaches to problems involving moving domain walls. In such cases, classical linear model order reduction methods may be ineffective due to the transport-dominated nature of the magnetization dynamics. Alternatively, one may use non-linear 
model order reduction techniques and neural networks to address these issues. 

Another possible development is the use of adaptive time-stepping to accurately capture rapid transitions in magnetization dynamics, such as switching events or domain wall motion. Refining the time discretization where sharp variations occur can improve accuracy at only a marginal increase in computational cost.

\section*{Acknowledgments}
AS's research was funded in part by the Austrian Science Fund (FWF) projects \href{https://doi.org/10.55776/F65}{10.55776/F65} and \href{https://doi.org/10.55776/P33477}{10.55776/P33477}.

MF and FH's work was funded by the Deutsche Forschungsgemeinschaft (DFG, German Research Foundation) – Project-ID 258734477 – SFB 1173 and the Austrian Science Fund (FWF) under the projects I6667-N and 
10.55776/F65 (SFB F65 “Taming complexity in PDE systems”). Funding was also received from the European Research Council (ERC) under the Horizon 2020 research and innovation program of the European Union (Grant agreement No. 101125225).

\section*{Declarations}
\textbf{Data Availability.} We do not analyze or generate any datasets because our work proceeds within a theoretical and mathematical approach.
\textbf{Conflicts of Interest.} The authors declare that they have no conflict of interest.
\textbf{Reproducibility.} Code available upon request.

\bibliographystyle{acm}
\bibliography{references}
\end{document}